\documentclass[oneside,11pt, a4paper,reqno]{amsart}
\usepackage[pdftex]{graphicx}
\graphicspath{{\main/pictures/}{pictures/}}
\usepackage{float}
\usepackage[T1]{fontenc}
\usepackage[utf8]{inputenc}
\usepackage{lmodern}
\usepackage[alphabetic,nobysame,abbrev]{amsrefs}%%This one not in the main lecture notes code
\usepackage[british]{babel}
\usepackage[babel, final]{microtype}
\usepackage{amssymb}
\usepackage{mathtools}

\usepackage[pdftex, dvipsnames]{xcolor}
\definecolor{darkblue}{rgb}{0,0,0.6}

\usepackage{caption}
\usepackage{subcaption}

\usepackage{ifpdf}
\usepackage{multirow}
\usepackage{adjustbox}
\usepackage{faktor} % for nicer quotients
\usepackage[shortlabels]{enumitem}

\usepackage{pinlabel}
\usepackage[breaklinks,%
    pdftex,%
    final,%
    colorlinks=true,%
    linkcolor=NavyBlue,%
    citecolor=blue,% DK: I find it better if there are 2 diff colors for link and cite.
    filecolor=NavyBlue,%
    menucolor=NavyBlue,%
    urlcolor=NavyBlue,%
    bookmarks=true,%
    bookmarksdepth=2,%
    bookmarksnumbered=true,%
    bookmarksopen=true,%
    bookmarksopenlevel=2,%
    unicode
]{hyperref}

\usepackage{overpic}
\usepackage{tikz,tikz-cd}
\usetikzlibrary{tikzmark}
\tikzset{
    math to/.tip={Glyph[glyph math command=rightarrow]},
    loop/.tip={Glyph[glyph math command=looparrowleft, swap]},
    weird/.tip={Glyph[glyph math command=Rrightarrow, glyph length=1.5ex]},
    pi/.tip={Glyph[glyph math command=pi, glyph length=1.5ex, glyph axis=0pt]},
}

\usepackage{imakeidx}
\makeindex

\usepackage{todonotes}
\usepackage{marginnote}

\makeatletter
\g@addto@macro\@floatboxreset\centering
\makeatother

\makeatletter
\def\l@figure{\@tocline{0}{3pt plus2pt}{0pt}{2.5pc}{}}
\makeatother

\allowdisplaybreaks[4]

\expandafter\let\csname ver@amsthm.sty\endcsname\relax
\let\theoremstyle\relax

\usepackage{amsthm}
\usepackage[capitalize, noabbrev]{cleveref}

\begingroup
    \makeatletter
    \@for\theoremstyle:=definition,remark,plain\do{%
        \expandafter\g@addto@macro\csname th@\theoremstyle\endcsname{%
            \addtolength\thm@preskip\parskip
            }%
        }
\endgroup

\newtheorem{theorem}{Theorem}[section]
\newtheorem{conjecture}[theorem]{Conjecture}
\newtheorem{corollary}[theorem]{Corollary}

\newtheorem{proposition}[theorem]{Proposition}
\newtheorem{lemma}[theorem]{Lemma}

\newtheorem*{rep@theorem}{\rep@title}
\newcommand{\newreptheorem}[2]{
\newenvironment{rep#1}[1]{
\def\rep@title{#2 \ref{##1}}
\begin{rep@theorem}}
{\end{rep@theorem}}}
\newreptheorem{theorem}{Theorem}
\newreptheorem{lemma}{Lemma}
\newreptheorem{proposition}{Proposition}
\newreptheorem{corollary}{Corollary}

\theoremstyle{definition}

\newtheorem{definition}[theorem]{Definition}
\newtheorem{construction}[theorem]{Construction}

\newtheorem{question}[theorem]{Question}

\newtheorem{exercise-easy}{\green{Exercise~$\triangle$}}
\newtheorem{exercise-medium}[exercise-easy]{\orange{Exercise~$\square$}}
\newtheorem{exercise-hard}[exercise-easy]{\red{Exercise~$\bigcirc$}}

\theoremstyle{remark}
\newtheorem{remark}[theorem]{Remark}

\numberwithin{equation}{section}
\makeatletter
\newcommand*\bigcdot{\mathpalette\bigcdot@{.7}}
\newcommand*\bigcdot@[2]{\mathbin{\vcenter{\hbox{\scalebox{#2}{$\m@th#1\bullet$}}}}}
\makeatother

\usepackage{subfiles}
\newcommand{\ol}{\overline}

\newcommand{\R}{\mathbb{R}}

\newcommand{\Z}{\mathbb{Z}}
\newcommand{\NN}{\mathbb{N}}

\DeclareMathOperator{\Int}{Int} %interior of a manifold
\DeclareRobustCommand\sm{\mathbin{\mathpalette\smaux\relax}}
\newcommand\smaux[2]{\mspace{-4mu}
\raisebox{\rsmraise{#1}\depth}{\rotatebox[origin=c]{-25}{$#1\smallsetminus$}}
 \mspace{-4mu}
}
\newcommand\rsmraise[1]{%
  \ifx#1\displaystyle .5\else
    \ifx#1\textstyle .5\else
      \ifx#1\scriptstyle .3\else
        .45%
      \fi
    \fi
  \fi}
\DeclareMathOperator{\ks}{ks}   % Kirby-Siebenman
\DeclareMathOperator{\Wh}{Wh}   % Whitehead group
\DeclareMathOperator{\std}{std} % standard
\makeatletter
\newcommand{\colim@}[2]{%
  \vtop{\m@th\ialign{##\cr
    \hfil$#1\operator@font colim$\hfil\cr
    \noalign{\nointerlineskip\kern1.5\ex@}#2\cr
    \noalign{\nointerlineskip\kern-\ex@}\cr}}%
}
\newcommand{\colim}{%
  \mathop{\mathpalette\colim@{\rightarrowfill@\textstyle}}\nmlimits@
}
\makeatother

\usepackage{letltxmacro}
\LetLtxMacro\Oldfootnote\footnote

\usepackage{chngcntr}
\counterwithin*{step}{theorem} 

\usepackage{enumitem}
\usepackage[dvipsnames]{xcolor}

\newcommand{\red}{\textcolor{WildStrawberry}}
\newcommand{\green}{\textcolor{OliveGreen}}
\newcommand{\orange}{\textcolor{orange}}

\usepackage{amsmath, amsthm, amssymb,mathtools}
\usepackage[margin=25mm]{geometry}

\newcommand{\homeo}{\approx}
\newcommand{\diffeo}{\cong}

\newcommand{\cone}{\mathrm{cone}}

\renewcommand{\std}{\mathrm{std}}
\newcommand{\C}{\mathcal{C}}
\newcommand{\T}{\mathcal{T}}
\newcommand{\F}{\mathcal{F}}
\newcommand{\lk}{\ell k}
\renewcommand{\P}{\mathcal{P}}
\newcommand{\N}{\mathcal{N}}
\newcommand{\B}{\mathcal{B}}
\newcommand{\AC}{\mathcal{AC}}
\newcommand{\Arf}{\operatorname{Arf}}
\renewcommand{\top}{\mathrm{top}}%for concordance groups
\newcommand{\diff}{\mathrm{diff}}%for concordance groups

\begin{document}
\title[Slice knots and knot concordance]{Slice knots and knot concordance}

\author{Arunima Ray}
\address{Max-Planck-Institut f\"{u}r Mathematik, Vivatsgasse 7, 53111 Bonn, Germany}
\email{\href{mailto:aruray@mpim-bonn.mpg.de}{aruray@mpim-bonn.mpg.de}}
\urladdr{\href{http://people.mpim-bonn.mpg.de/aruray/}{http://people.mpim-bonn.mpg.de/aruray/}}

\begin{abstract}
These notes were prepared to accompany a sequence of three lectures at  the conference Winterbraids XI in Dijon, held in December 2021. In them, we provide an introduction to slice knots and the equivalence relation of concordance. We explain some connections between slice knots and exotic smooth structures on $\R^4$. We also introduce filtrations of the knot concordance groups and satellite operations. 
\end{abstract}
\maketitle

\section*{Overview}
Slice knots were first defined in 1958 by Fox and Milnor and have since become a flourishing field of study. While they were originally considered in the context of resolving singularities of surfaces in $4$-manifolds, numerous other connections to questions in $3$- and $4$-manifold topology have been discovered. A notable highlight: every knot which is topologically slice but not smoothly slice gives rise to an exotic smooth structure on $\R^4$; more on this in \cref{sec:ex-sm-str}. 

The goal of these lecture notes is to provide an overview of the basic notions in the study of slice knots and knot concordance, leading up to a small selection of recent developments. Due to the limitations of time we will barely scratch the surface of this active and vibrant area. Nevertheless we will attempt to provide some references and pointers to other resources.
%The latest version of these notes can be found at \href{http://www.tinyurl.com/slicenotes}{www.tinyurl.com/slicenotes}. 

\subsection*{Exercises}
There are exercises throughout the lecture notes, which are also compiled in a list at the end. The problems are separated into three levels. \textbf{\green{Green $\triangle$}} exercises are usually straightforward and should be attempted if you are seeing this material for the first time. Prerequisites are courses in introductory geometric and algebraic topology. \textbf{\orange{Orange $\square$}} exercises are for readers who are already comfortable with some of the terminology; they may require nontrivial input from outside these lectures, which we have tried to indicate as hints. Finally, \textbf{\red{red $\bigcirc$}} exercises are challenge problems. Open problems will be marked as such, and do not intersect with the exercises.

\subsection*{Relationship to lectures} Some details and exercises in these notes were not mentioned in the accompanying lectures. The order of topics has also been slightly modified. Many of the exercises are new.

\subsection*{Conventions}  Homeomorphism of manifolds is denoted by the symbol $\homeo$, while diffeomorphism is denoted by $\diffeo$. All manifolds are assumed to be oriented. All knots are assumed to be tame and oriented, i.e.~(images of) smooth embeddings $S^1\hookrightarrow S^3$, where both the domain and codomain are oriented by hypothesis. The set of nonnegative integers is denoted by $\NN$.

\subsection*{Acknowledgements}
I am grateful to the organisers of Winterbraids XI in Dijon for an excellent conference, especially considering the Covid-19 pandemic, and to the attendees for their lively participation and many questions. Many thanks are due as well to the anonymous referee for their helpful comments.

\section{Slice and ribbon knots}\label{sec:defns}
Recall that a knot $K\colon S^1\hookrightarrow S^3$ is \emph{trivial} or the \emph{unknot} if and only if it bounds an embedded disc in $S^3$. We will often conflate a knot and its image, and it will be important that both $S^1$ and $S^3$ are oriented. See~\cref{fig:exknots} for some examples of knots. Here, by definition, a knot is trivial if it can be deformed via an ambient isotopy to the round unit circle $S^1\subseteq \R^3\subseteq S^3$. For more on classical knot theory, see e.g.~\cites{fox-quicktrip,rolfsen-book,livingston-book}. 

\begin{figure}[htb]
\centering
\includegraphics[width=15cm]{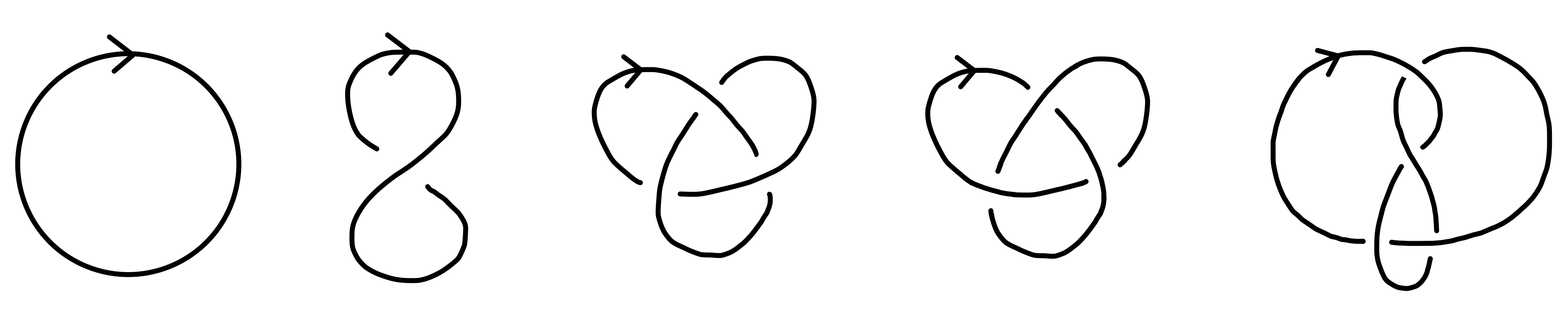}
\caption{The two knots on the left are trivial. The three knots on the right are not trivial, and are respectively called the \emph{right-handed trefoil}, the \emph{left-handed trefoil}, and the \emph{figure eight} knot. 
}\label{fig:exknots}
\end{figure} 

Slice knots are, in some sense, generalised trivial knots -- we still ask for the knot to bound an embedded disc, but it has the freedom of an additional (half-)dimension.

\begin{definition}
A knot $K\subseteq S^3=\partial B^4$ is \emph{smoothly slice} if it bounds a smoothly embedded disc, called a \emph{smooth slice disc}, in $B^4$.\footnote{In 1958, Fox and Milnor called these \emph{null-equivalent} knots, in unpublished work. Fox adopted the term \emph{slice knot}, proposed by Edward Moise, in~\cite{fox-quicktrip}.}
\end{definition}

We begin by discussing a number of elementary ways to construct new slice knots from a given slice knot. Let $K\colon S^1\hookrightarrow S^3$ be a knot. The knot obtained by reversing the orientation of $S^1$ is called the \emph{reverse} of $K$, denoted by $rK$. The knot obtained by reversing the orientation of $S^3$ is called the \emph{mirror image} of $K$, denoted by $\ol{K}$. For example, the left-handed trefoil is the mirror image of the right-handed trefoil. Both trefoils are isotopic to their reverses. The figure eight is isotopic to both its reverse and its mirror image. 

\begin{proposition}[\green{Exercise~$\triangle$}~\ref{ex:reverse-mirror-slice}]\label{prop:reverse-mirror-slice}
Let $K\subseteq S^3$ be a knot. Show that $K$ is smoothly slice if and only if $rK$ is smoothly slice if and only if $\ol{K}$ is smoothly slice.
\end{proposition}

Given knots $K,J\subseteq S^3$, we can form their \emph{connected sum} by taking the (oriented!) connected sum of pairs $(S^3,K)\# (S^3,J)$; see~\cref{fig:conn-sum}. Given diagrams of knots $K$ and $J$ we can form a diagram of the connected sum by connecting the two diagrams by a trivial band, as shown in the figure, taking care to match the original orientations. This is a special case of a \emph{band sum} defined later in this section.

\begin{figure}[htb]
\centering
\includegraphics[width=6cm]{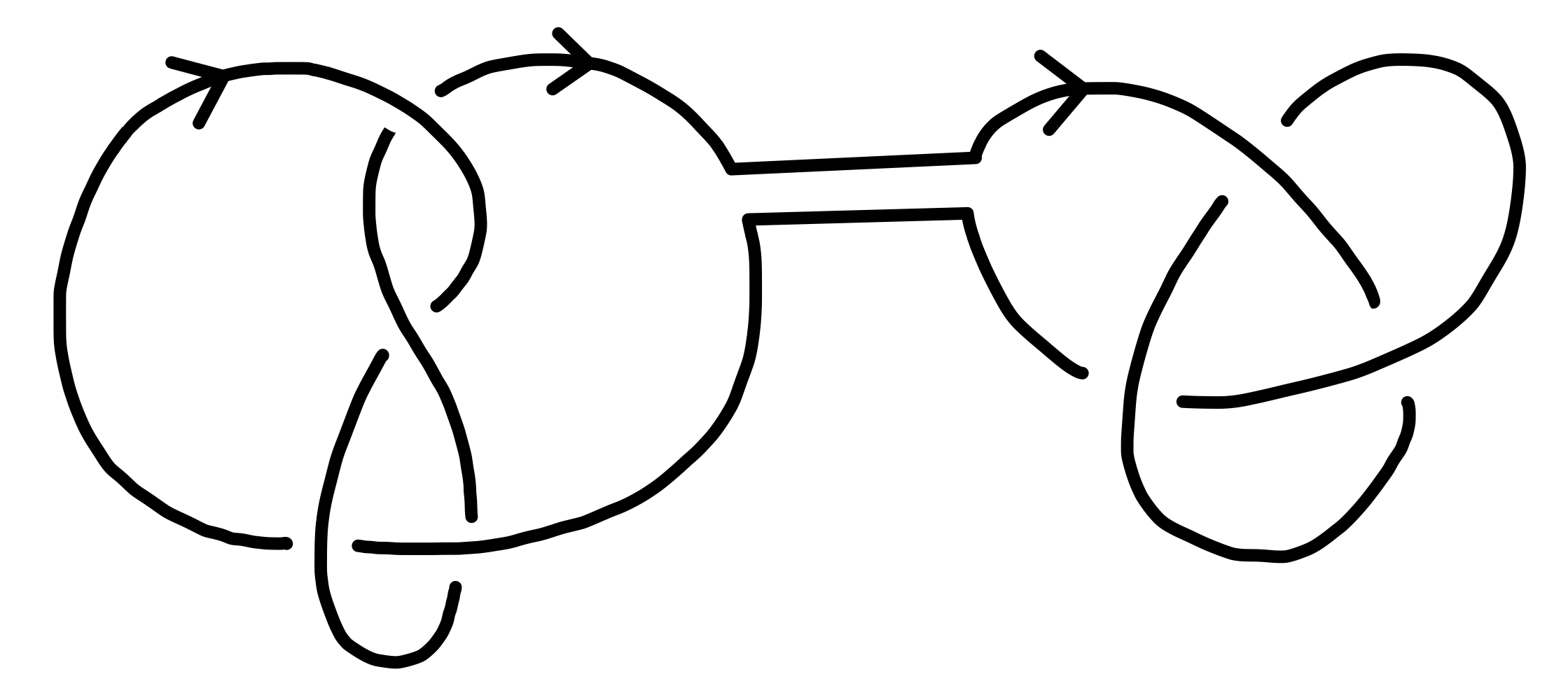}
\caption{The connected sum of the figure eight and the right-handed trefoil from \cref{fig:exknots}.}\label{fig:conn-sum}
\end{figure} 

\begin{proposition}[\green{Exercise~$\triangle$}~\ref{ex:conn-sum-slice}]\label{prop:conn-sum-slice}
If the knots $K,J\subseteq S^3$ are smoothly slice, then so is $K\# J$. 
\end{proposition}

By ambient Morse theory, we can isotope any smooth slice disc so that the level sets with respect to the radial function on $B^4$ are links, i.e.~(the image of) a smooth embedding of a disjoint union of circles in $S^3$, except at finitely many radii where the level sets have singularities. Since we are considering a disc in a $4$-manifold, there are only three possible types of singularities: minima, saddles, and maxima. See~\cref{fig:slice-schematic} for a picture and \citelist{\cite{milnor-hcobtheorem}\cite{milnor-morsetheory}\cite{matsumoto-morsetheory}\cite{nicolaescu-morsetheory}\cite{gompf-stipsicz-book}*{Sections 4.2 and 6.2}} for more on Morse theory.

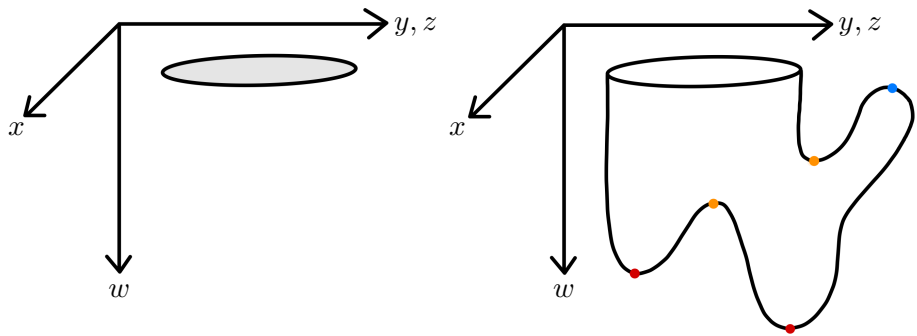
\begin{figure}[htb]
	\centering
\begin{tikzpicture}
        \node[anchor=south west,inner sep=0] at (0,0){	\includegraphics[width=12cm]{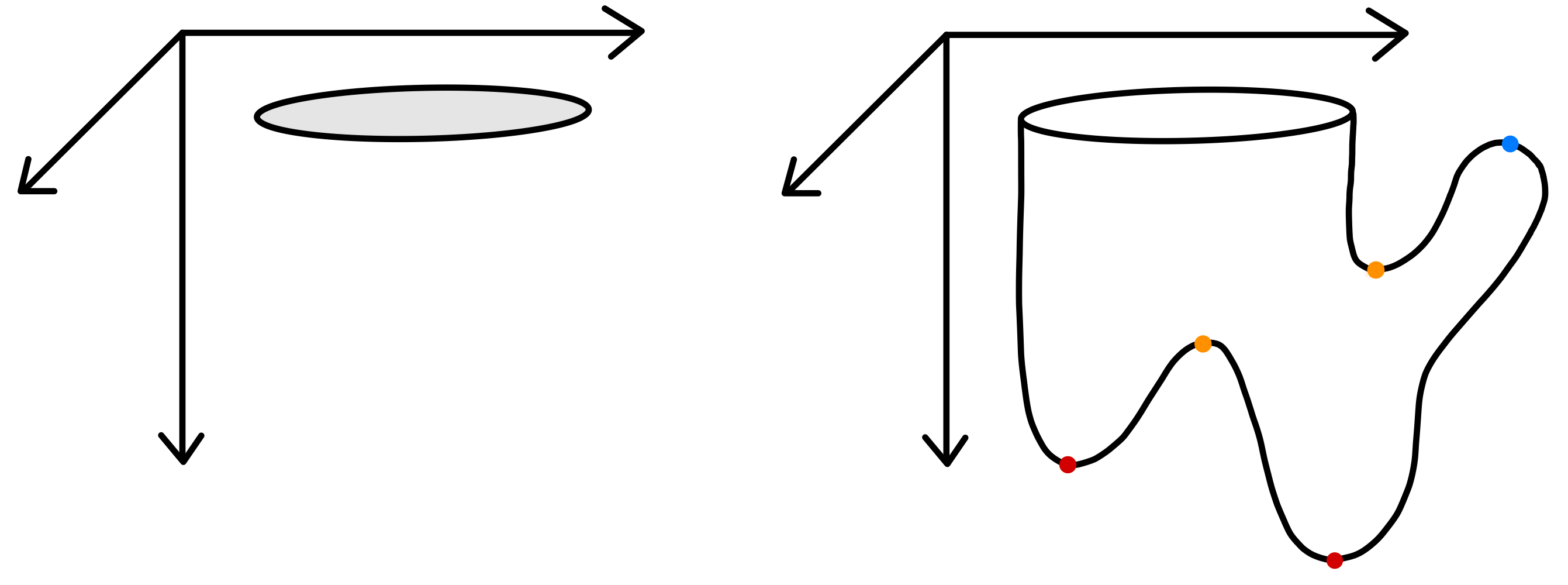}};
		%\draw[step=1cm,color=gray] (0,0) grid (12,5);
		\node at (0.05,2.85) {$x$};
		\node at (5.3,4.2) {$y,z$};
		\node at (1.4,0.7) {$w$};
		\node at (5.85,2.85) {$x$};
		\node at (11.15,4.2) {$y,z$};
		\node at (7.25,0.7) {$w$};
	\end{tikzpicture}
\caption{The axes schematically indicate $B^4$, where the coordinates $x,y,z$ describe $\R^3\subseteq S^3$ and the coordinate $w\geq 0$ indicates depth, so that $S^3$ occurs at $w=0$. In both cases we have a knot in $\R^3\subseteq S^3=\partial B^4$. On the left we have a trivial knot, as detected by the embedded disc in $S^3$ shaded grey. On the right we have a smoothly slice knot, with a smooth slice disc lying in $B^4$. The disc is in a particularly nice position with respect to the radial function on $B^4$. At most $w$-slices, the disc appears as a link, i.e.\ an embedding of a disjoint union of circles. At finitely many values of $w$, we see singularities. There are three types of singularities: minima (red), saddles (orange), and maxima (blue). See~\cref{fig:ribbon-movie} for a more precise diagram.}\label{fig:slice-schematic}
\end{figure} 

Smooth slice discs, once they have been perturbed to respect the radial Morse function as above, can be described in terms of movies. More specifically, we consider the radial level sets as the radius increases from $0$ to $1$, with the latter corresponding to $\partial B^4=S^3$ containing a smoothly slice knot. At each minimum, an unknot is born, split from every other component, if any. At every saddle singularity, two portions of a link either merge together, or split apart. At a maximum, an unknotted component, split from everything else, disappears. See~\cref{fig:ribbon-movie} for an example.  

\begin{figure}[htb]
\centering
\includegraphics[width=15cm]{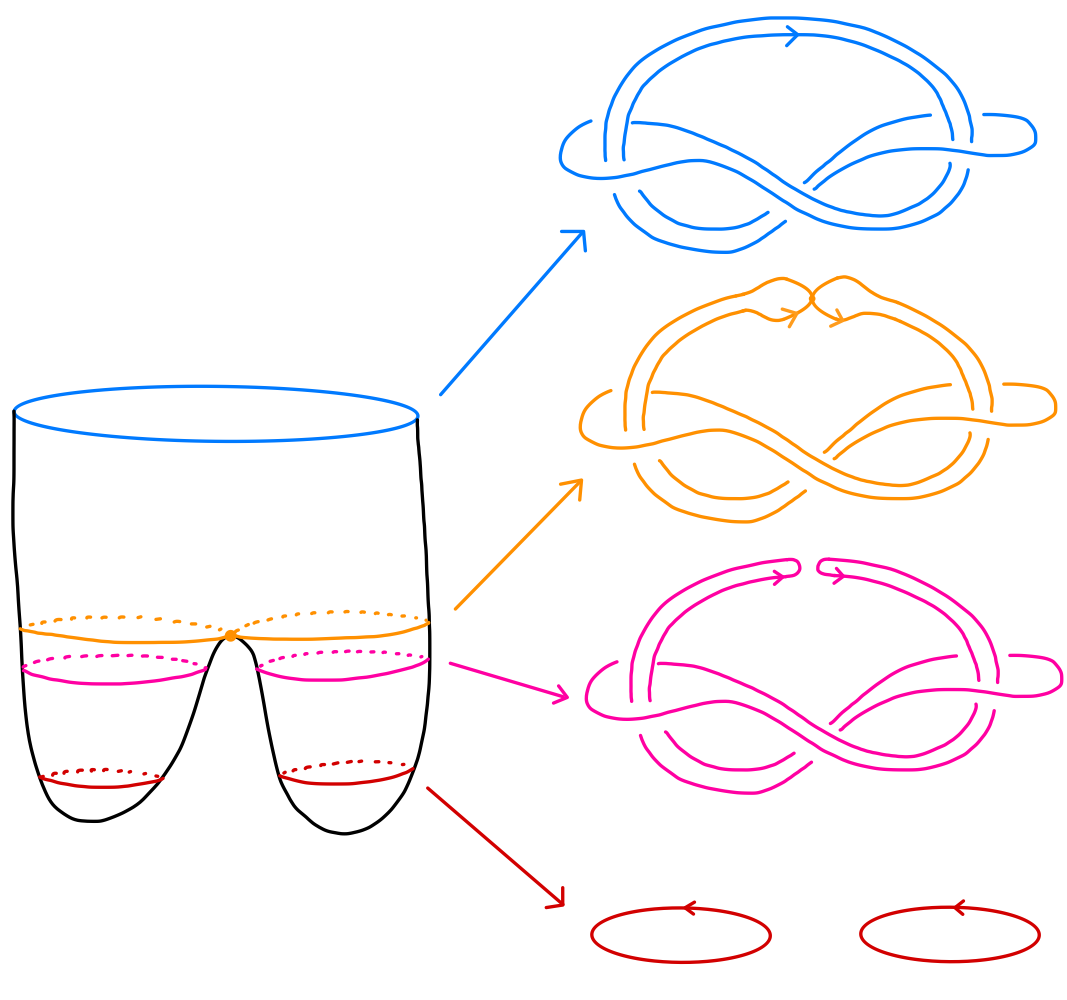}
\caption{Describing a ribbon disc as a movie. At the very bottom, we have two minima. At a slightly larger radius in $B^4$, the minima give rise to (split) unknots, i.e. an unlink. The subsequent cross sections, until the orange saddle singularity, describe an isotopy of this unlink. At the saddle singularity, two portions of the previous unlink touch each other, and at a slightly higher radius we see a ribbon knot.}\label{fig:ribbon-movie}
\end{figure} 

We especially like smooth slice discs with only two types of singularities, as in the following definition. 

\begin{definition}\label{def:ribbon}
A knot $K\subseteq S^3=\partial B^4$ is ribbon if it bounds a smoothly embedded disc in $B^4$ with only minima and saddles. Such a disc is called a \emph{ribbon} disc.
\end{definition}

\begin{figure}[htb]
	\centering
\begin{tikzpicture}
        \node[anchor=south west,inner sep=0] at (0,0){	\includegraphics[width=7cm]{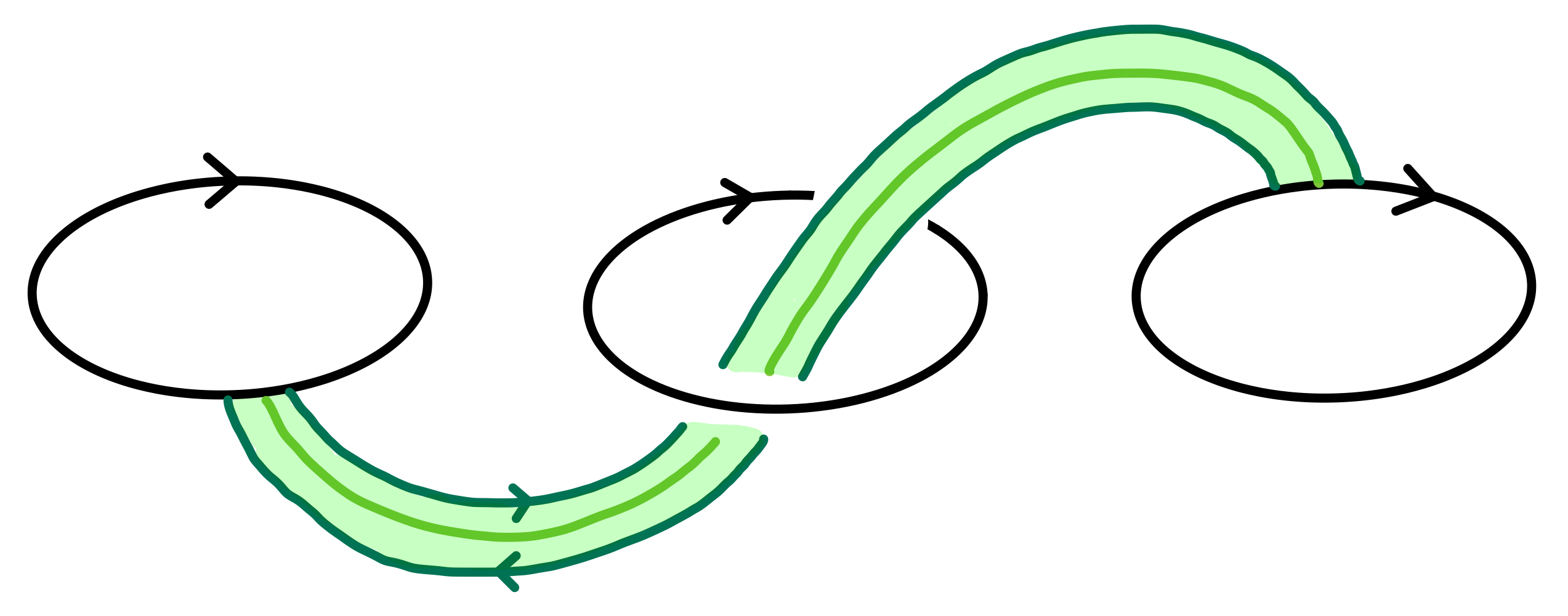}};
        \node[anchor=south west,inner sep=0] at (8,0){	\includegraphics[width=7cm]{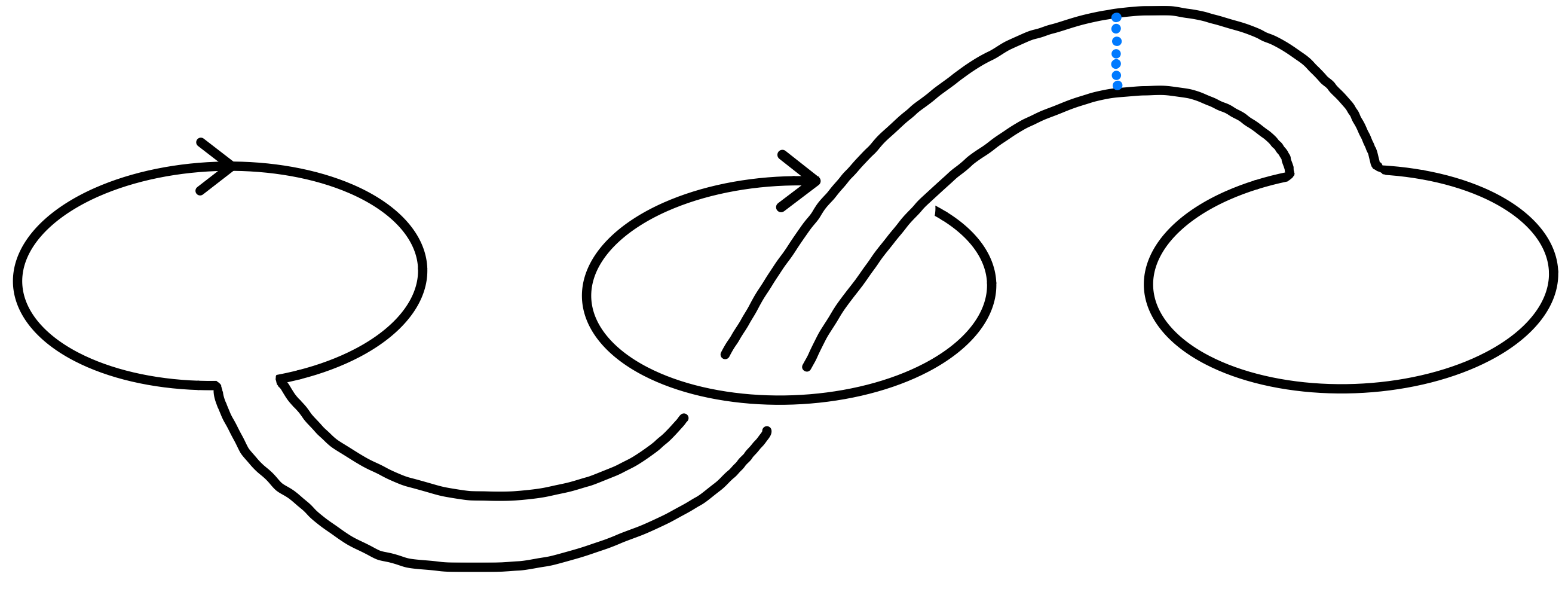}};
		%\draw[step=1cm,color=gray] (0,0) grid (15,2.5);
		\node at (-0.15,1.2) {$L$};
		\node at (2.35, 0.76) {$B$};
		\node at (3.5,-0.5) {(a)};
		\node at (11.5,-0.5) {(b)};
	\end{tikzpicture}
\caption{(a) A link $L$ (black) and a band $B$ (dark green). The core of $B$ is shown in light green. (b) The result of the band sum on $L$ along $B$ is shown in black. The arc guiding the dual band move reversing the previous band move is shown in blue. }\label{fig:band-sum}
\end{figure}

Given a link $L\subseteq S^3$, a band $B$ is a copy of $[0,1]\times[0,1]$ embedded in $S^3$, with $\{0,1\}\times [0,1]$ lying on $L$ (matching the orientations), and otherwise disjoint from $L$. The arc $[0,1]\times \{1/2\}$ is called the \emph{core} of $B$. A \emph{band sum} of $L$ along the band $B$ is the result of removing $\{0,1\}\times [0,1]$ from $L$ and gluing in $[0,1]\times \{0,1\}$; see \cref{fig:band-sum}. When the components of $\{0,1\}\times [0,1]$ lie on distinct components of $L$ the result is also called a \emph{fusion} of $L$ along $B$. Note that the core $[0,1]\times \{1/2\}$ of a band can be an interesting arc in $S^3\sm L$, and in particular can link with the components of $L$. A band sum on $L$ along $B$ can be reversed by performing a dual band move, and corresponds to removing $[0,1]\times \{0,1\}$ from the band sum and gluing $\{0,1\}\times [0,1]$ back in. One notices that the result is isotopic to the original link $L$. 

From the movie perspective, reading the ribbon disc from the bottom up, we see that any ribbon knot is a fusion of some trivial link. (Why is it a fusion, rather than an arbitrary band sum?) In other words, it can be obtained from an $n$-component trivial link for some $n$ (corresponding to minima) and a collection of $n-1$ bands (corresponding to saddles), by performing band sums, each of which joins two distinct components of the original link, and no two of which join the same pair of components. Alternatively, we could run the movie backwards. From this perspective, a ribbon knot is a knot where one can add $n-1$ bands for some $n$ (corresponding to the dual band moves described above), so that the result is an $n$-component unlink. 

The following shows that ribbon knots can be described purely $3$-dimensionally. This gives a useful method to detect if a given knot is ribbon.

\begin{proposition}[\orange{Exercise~$\square$}~\ref{ex:ribbon-sing}]\label{prop:ribbon-sing}
A knot $K\subseteq S^3$ is ribbon if and only if it bounds a disc in $S^3$ with only \emph{ribbon singularities}, i.e.\ singularities of the form shown in \cref{fig:ribbon-sing}.\footnote{Such singular discs in $S^3$ are also sometimes called ribbon discs.}
\end{proposition}

\begin{figure}[htb]
\centering
\includegraphics[width=15cm]{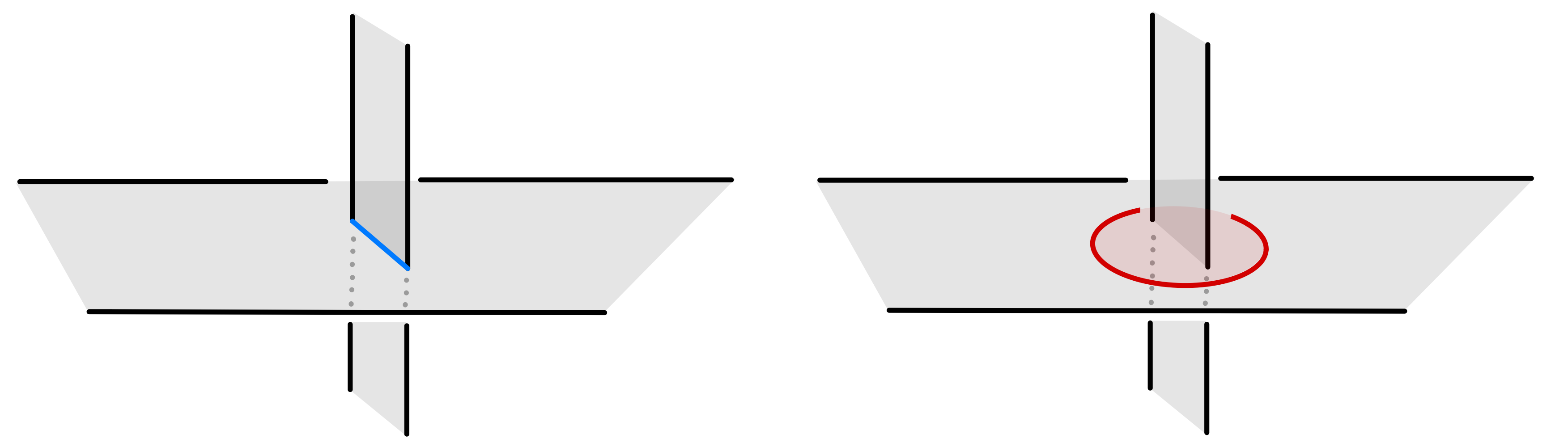}
\caption{Left: A \emph{ribbon singularity} (shown in blue) of a disc (shown in grey) in $S^3$. Right: The smaller disc region encircled in red on the horizontal sheet could be pushed radially into $B^4$ to resolve the singularity.}\label{fig:ribbon-sing}
\end{figure} 

\begin{proposition}[\green{Exercise~$\triangle$}~\ref{ex:easy-ribbon}]\label{prop:easy-ribbon}
For any knot $K\subseteq S^3$, the knot $K\# r\ol{K}$ is ribbon.
\end{proposition}

From \cref{def:ribbon}, one sees directly that any ribbon knot is smoothly slice. (For the equivalent definition given in \cref{prop:ribbon-sing}, one could push a small region of the ribbon disc in $S^3$ into $B^4$ to produce an embedded disc, as indicated in~\cref{fig:ribbon-sing}.) The converse is an important open problem.

\begin{conjecture}[Open, slice-ribbon conjecture, Fox~\cite{fox-problems}*{Problem~25}]
Every smoothly slice knot is ribbon.
\end{conjecture}

This conjecture has been established for certain families of knots, e.g.\ $2$-bridge knots~\cite{lisca-2bridge} and most $3$-strand pretzel knots~\cites{greene-jabuka-pretzel,lecuona-pretzel}, but there is no general approach to attacking it at the moment. Other work on the slice-ribbon conjecture using a similar strategy includes \cites{lecuona-montesinos,bryant-sliceribbon,long-sliceribbon,AKPR-sliceribbon}. Some potential counterexamples have been proposed, e.g.~\cites{gst-sliceribbon,abetagami-sliceribbon}. A very promising candidate had been the $(2,1)$-cable of the figure eight, which was known to be nonribbon for twenty years~\cite{miyazaki-cable} but whose sliceness status was unknown. It was recently shown to not be smoothly slice~\cite{21cable-first} (see also~\cite{21cable-second}). Notably it is still open whether it is topologically slice (\cref{def:top-slice}).

The subtlety of the conjecture is seen by the following proposition. Specifically, to show that a smoothly slice knot is ribbon it is not a good strategy to prove that a given smooth slice disc is ribbon. 

\begin{proposition}[\orange{Exercise~$\square$}~\ref{ex:bad-disc}]\label{prop:bad-disc}
There exist smooth slice discs that are not ambiently isotopic (relative to the boundary) to any ribbon disc. 
\end{proposition}

\subsection{Topologically slice knots}
So far we have worked strictly in the smooth setting. We can however loosen this restriction slightly. We will need the following definition. 

\begin{definition}
An embedding $\varphi\colon (F,\partial F)\hookrightarrow (M,\partial M)$, i.e.~a continuous map which is a homeomorphism onto its image, of a surface $F$ in a $4$-manifold $M$ is said to be \emph{locally flat} if for all $x\in F$ there is a neighbourhood $U$ of $\varphi(x)$ such that $(U,U\cap \varphi(F))$ is homeomorphic to either $(\R^4,\R^2)$, in the case that $x\in \Int{F}$, or to $(\R^4_+, \R^2_+)$, in the case that $x\in \partial F$. 
\end{definition}

The above gives rise to another notion of sliceness for knots. 

\begin{definition}\label{def:top-slice}
A knot $K\subseteq S^3=\partial B^4$ is \emph{topologically slice} if it bounds a locally flat embedded disc, called a \emph{topological slice disc}, in $B^4$. 
\end{definition}

It is straightforward to see that any smoothly slice knot is topologically slice. Via a deep result of Quinn~\citelist{\cite{quinn-endsiii}*{Theorem~2.5.1}\cite{FQ}*{Theorem~9.3}}, a topological slice disc admits a tubular neighbourhood, i.e. we could equivalently define a knot $K$ to be topologically slice if it bounds an embedded disc $\Delta$ in $B^4$, admitting a neighbourhood homeomorphic to $\Delta\times D^2$, intersecting $S^3$ in a tubular neighbourhood of $K$.  

\begin{remark}
We motivated slice knots by saying that they generalise trivial knots. Certainly the trivial knot is topologically slice. We should however be careful not to be too general. Given any knot $K\subseteq S^3$, the coned disc $\cone(K)\subseteq \cone(S^3)=B^4$ is an embedded disc bounded by $K$ in $B^4$. See \cref{fig:cone}. As shown in the following proposition, this disc is not in general locally flat. 
\end{remark}

\begin{figure}[htb]
\centering
\includegraphics[width=5cm]{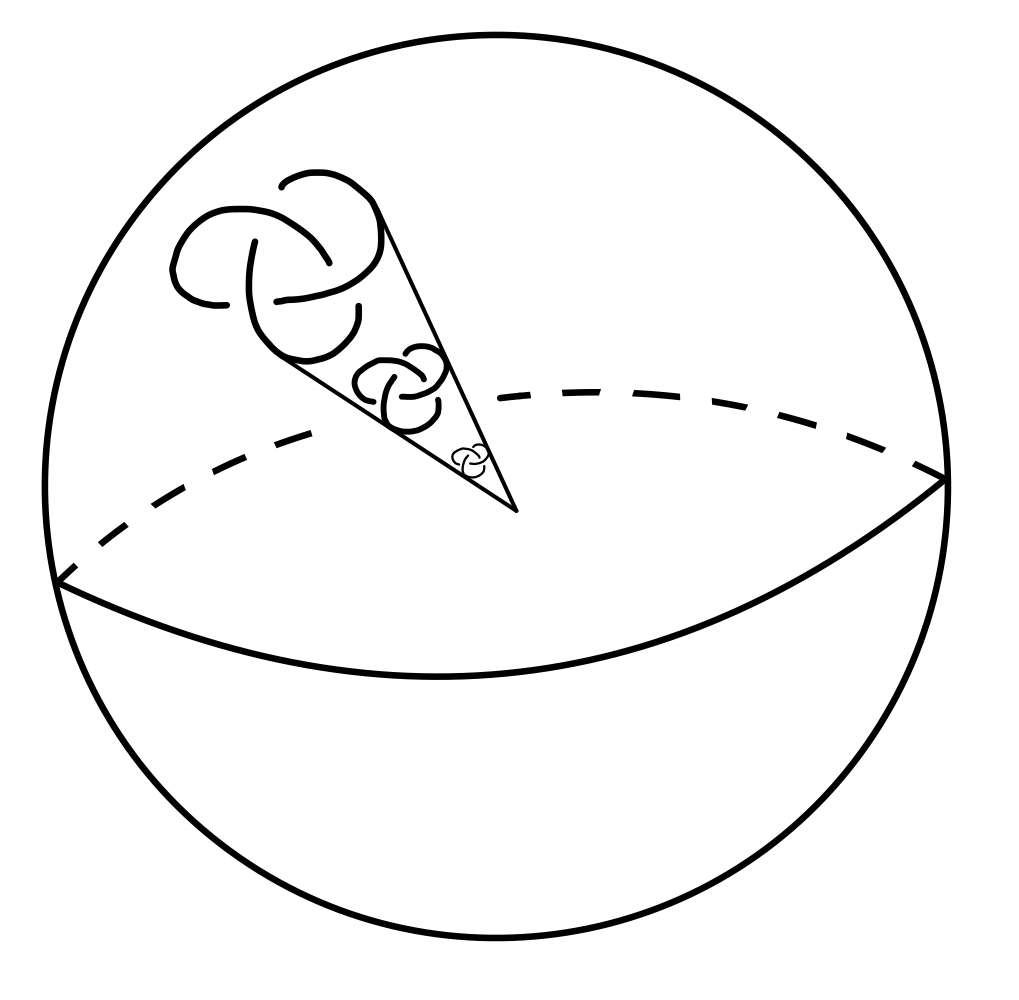}
\caption{Let $K\subseteq S^3$ be a knot. Thinking of $B^4$ as the cone on $S^3$, we get an embedded disc $\cone(K)\subseteq B^4$.}\label{fig:cone}
\end{figure}

\begin{proposition}[\green{Exercise~$\triangle$}~\ref{ex:cone}]\label{prop:cone}
Let $K\subseteq S^3$ be a knot. The coned disc $\cone(K)\subseteq \cone(S^3)=B^4$ is locally flat if and only if $K$ is the trivial knot.
\end{proposition}

The following result of Freedman and Quinn gives a powerful method to construct topologically slice knots. For the definition of the Alexander polynomial, see e.g.~\cite{rolfsen-book}*{Chapter~8}. (See~\cite{feller-topslicegenus} for a generalisation to the topological slice genus.)

\begin{theorem}[\citelist{\cite{FQ}*{Theorem~11.7B}\cite{garoufalidis-teichner}}]
Let $K\subseteq S^3$ be a knot. If the Alexander polynomial $\Delta_K(t)\doteq 1$, then $K$ is topologically slice.
\end{theorem}

As a consequence of the above theorem, the (untwisted) Whitehead double of any knot (\cref{fig:whitehead-doubling}) is topologically slice, since it has Alexander polynomial one. On the other hand, many such knots are not smoothly slice.

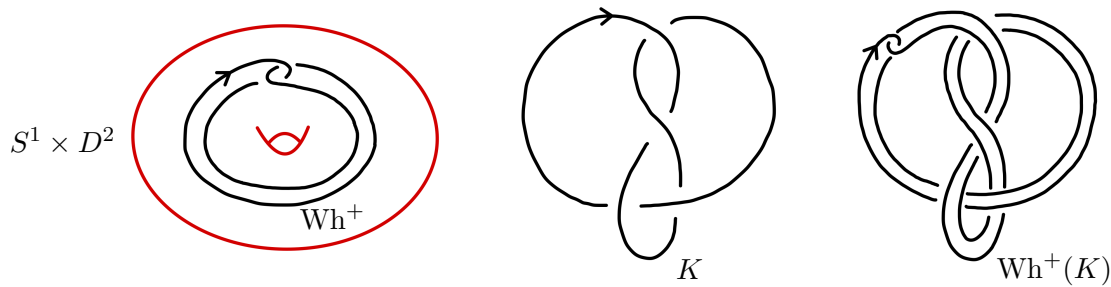
\begin{figure}[htb]
	\centering
\begin{tikzpicture}
        \node[anchor=south west,inner sep=0] at (0,0){%	
        \includegraphics[width=13cm]{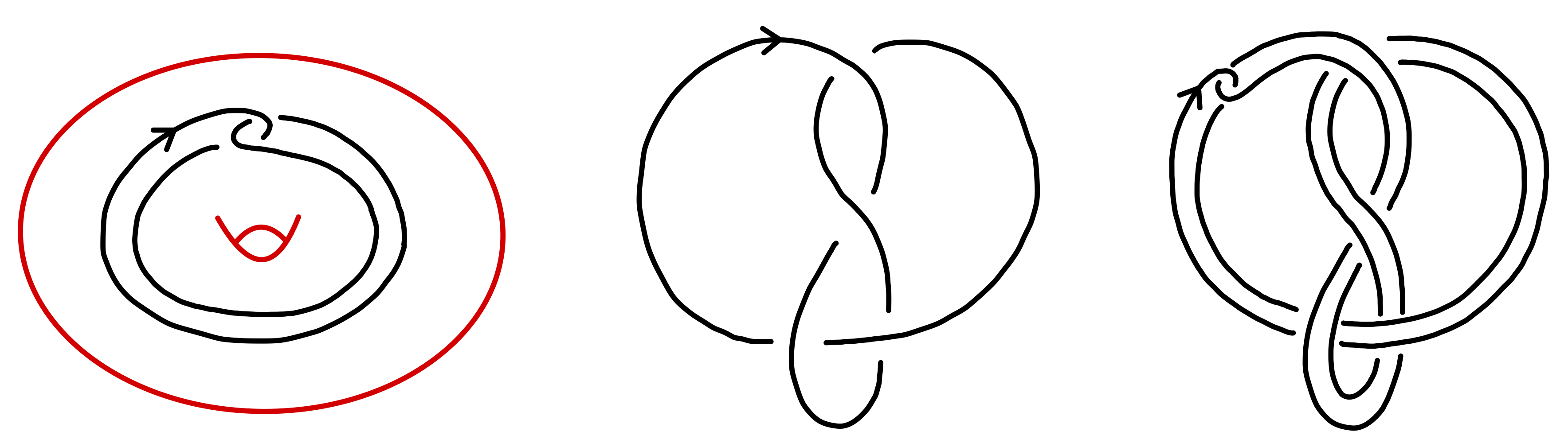}};%
		%\draw[step=1cm,color=gray] (0,0) grid (13,4);
		\node at (-0.75,1.7) {$S^1\times D^2$};
		\node at (2.8, 0.7) {$\Wh^+$};
		\node at (7.5,0) {$K$};
		\node at (12.3,0) {$\Wh^+(K)$};
	\end{tikzpicture}
\caption{Whitehead doubling. On the left we have a knot $\Wh^+$ in a solid torus $S^1\times D^2$. In the middle we show the figure eight knot. Given an arbitrary knot $K$, the (positive clasped, untwisted) \emph{Whitehead double} of $K$, denoted by $\Wh^+(K)$, is the image of $\Wh^+$ under a map identifying $S^1\times D^2$ with a tubular neighbourhood of $K$, so that the longitude $S^1\times \{*\}$ of $S^1\times D^2$ is sent to the Seifert longitude of $K$. On the right we show the (positive clasped, untwisted) Whitehead double of the figure eight. The negative clasped Whitehead double $\Wh^-(K)$ is defined in the same way, using as input the knot $\Wh^-$ in $S^1\times D^2$ obtained from $\Wh^+$ by changing both crossings. Whitehead doubling is a special case of the \emph{satellite construction} on knots (\cref{def:satellite}). }\label{fig:whitehead-doubling}
\end{figure}

\begin{theorem}
There exist knots that are topologically slice but not smoothly slice, e.g.\ the positive clasped untwisted Whitehead double of the right-handed trefoil knot.
\end{theorem}

The above was first shown by Casson and Akbulut, in independent unpublished work, using Donaldson's diagonalisation theorem~\cites{donaldson,donaldson-connections,donaldson-yangmills}. The first published accounts are by Gompf~\cite{gompf-smoothconc}, Cochran--Lickorish~\cite{cochran-lickorish}, and Cochran--Gompf~\cite{cochran-gompf}.

However, for many examples of Whitehead doubles, it is still open whether they are smoothly slice. For the next question, let $L$ denote the left-handed trefoil and let $4_1$ denote the figure eight. %Let $\Wh^\pm(K)$ denote, respectively, the positive or negative clasped, untwisted Whitehead double of a knot $K$. 

\begin{question}[Open]\label{ques:special-wh}
Is $\Wh^+(L)$ or $\Wh^\pm(4_1)$ smoothly slice?
\end{question}

It is not difficult to see that if the knot $K$ is smoothly slice, then so is $\Wh^\pm(K)$; see \green{Exercise~$\triangle$}~\ref{ex:wh-slice}. The converse is an interesting open question, generalising \cref{ques:special-wh}.

\begin{conjecture}[Open~\cite{kirbylist}*{Problem~1.38}]\label{conj:whitehead-injective}
Let $K\subseteq S^3$ be a knot. The knot $Wh^\pm(K)$ is smoothly slice if and only if $K$ is smoothly slice. 
\end{conjecture}

As mentioned earlier, Donaldson's diagonalisation theorem was used to show that the positive clasped untwisted Whitehead double of the right-handed trefoil is not smoothly slice. This is a relatively subtle sliceness obstruction, since e.g.\ it does not obstruct topological sliceness. 
There are numerous other slicing obstructions, including many topological sliceness obstructions. We will return to this topic in~\cref{sec:conc-gps-filtration}, where we will explain a method to arrange sliceness obstructions in order of strength, in some sense. We defer that discussion to first describe some connections between slice knots and $4$-manifold topology.

\subsection{Links} The notions in this section can be generalised to links, i.e.~embeddings $\sqcup S^1\hookrightarrow S^3$. For example, we have the following definition.

\begin{definition}
A link $L\subseteq S^3=\partial B^4$ is said to be \emph{(strongly) smoothly slice} if the components bound a collection of pairwise disjoint smooth slice discs in $B^4$. If the discs are only locally flat, the link is said to be \emph{(strongly) topologically slice}. 
\end{definition}

There is a parallel notion of \emph{weak} sliceness of links, in both the topological and smooth settings. For this we do not require a collection of pairwise disjoint embedded discs, but rather just some embedded planar surface (either smooth or locally flat). 

Unfortunately we will not have time to pursue links thoroughly. We limit ourselves to saying that the open questions mentioned in this section have analogues for links: it is open whether the slice-ribbon conjecture is true for links and whether the Whitehead double of a link $L$ is smoothly slice if and only if $L$ is smoothly slice. It is also open whether the Whitehead double of every link is topologically slice, and an answer would have deep consequences for fundamental open questions in 4-manifold topology (see e.g.~\cite{DETbook-enigmata}).

\section{\texorpdfstring{Exotic smooth structures on $\R^4$}{Exotic smooth structures on R4}}\label{sec:ex-sm-str}

By definition, a smooth structure on a topological manifold $M$ is a maximal atlas of charts on $M$ such that the transition maps are smooth. \textit{A priori} a given topological manifold can admit finitely many (including none) or infinitely many smooth structures. The study of existence and uniqueness questions for smooth structures comprises the field of \emph{smoothing theory}. Notably, the sphere $S^7$ has 28 distinct smooth structures up to orientation preserving diffeomorphism~\cites{milnor-exotic,kervaire-milnor}, and whether $S^4$ has more than one smooth structure is the content of the (still open) smooth $4$-dimensional Poincar\'{e} conjecture. Smoothing theory is well developed in dimensions five and higher, and for example, using technology of Kirby--Siebenmann from~\cite{KS}, it can be shown that closed manifolds of dimension five and higher have finitely many (possibly zero) smooth structures. Smoothing closed $4$-manifolds is more complicated. There exist topological $4$-manifolds which do not admit any smooth structures, such as Freedman's $E8$-manifold~\cite{F}, while there are closed $4$-manifolds admitting infinitely many smooth structures, such as the $K3$ surface~\cite{fintushel-stern-knot-surgery}.

The existence and uniqueness questions for smooth structures can also be asked for noncompact manifolds. The following result shows that Euclidean spaces are well behaved assuming $n\neq 4$. 

\begin{theorem}[\cites{moise-book,stallings-euclidean}]
For $n\neq 4$ there is a unique smooth structure on $\R^n$, up to diffeomorphism.
\end{theorem}

By contrast, there are \textit{uncountably} many smooth structures on $\R^4$~\cite{taubes-uncountable}. 

\begin{definition}
Let $\R^4_\std$ denote $\R^4$ with its standard smooth structure. A smooth $4$-manifold which is homeomorphic to $\R^4_\std$, but not necessarily diffeomorphic to $\R^4_\std$, is called an \emph{$\R^4$-homeomorph}. An $\R^4$-homeomorph which is not diffeomorphic to $\R^4_\std$ is called an \emph{exotic} $\R^4$, and such a smooth structure is called an \emph{exotic} smooth structure on $\R^4$. 
\end{definition}

There are two distinct types of $\R^4$-homeomorphs, as per the following definition. 

\begin{definition}
Let $\mathcal{R}$ be an $\R^4$-homeomorph. The manifold $\mathcal{R}$ is said to be \emph{small} if it admits a smooth embedding $\mathcal{R}\hookrightarrow \R^4_\std$. It is called \emph{large} otherwise.\footnote{There are two slightly different definitions of small/large $\R^4$-homeomorphs in the literature. Ours coincides with the definition of Scorpan's book~\cite{scorpan-book}. In the book by Gompf--Stipsicz~\cite{gompf-stipsicz-book}, a large $\R^4$-homeomorph is one which contains a compact subset that does not embed in $\R^4_\std$, and an $\R^4$-homeomorph is called small if it is not large. It is not known whether the two definitions are equivalent.}
\end{definition}

Knots and links feature prominently in the known constructions of exotic $\R^4$s. In this section we describe two such constructions, one of large exotic $\R^4$s, and one of small exotic $\R^4$s.

\subsection{Large exotic \texorpdfstring{$\R^4$s}{R4s} using 0-traces}

For any knot $K$ we have the \emph{$0$-trace}, which is by definition the smooth $4$-manifold obtained by adding a $0$-framed $2$-handle to $B^4$ along $K\subseteq S^3$, and then smoothing corners. One of the two key ingredients in the forthcoming construction is the following characterisation of slice knots. 

\begin{lemma}[Trace embedding lemma, \orange{Exercise~$\square$}~\ref{ex:trace-embedding-lemma}]\label{lem:trace-embedding-lemma}
Let $K\subseteq S^3$ be a knot. The $0$-trace $X_0(K)$ admits a smooth (resp. locally collared) embedding into $\R^4_\std$ if and only if $K$ is smoothly (resp. topologically) slice.
\end{lemma}

We will also need the following powerful result.

\begin{theorem}[\citelist{\cite{quinn-endsiii}*{Corollary~2.2.3}\cite{FQ}*{Theorem~8.2, Section~8.7}}]\label{thm:quinn-smoothing}
Let $M$ be a connected, noncompact $4$-manifold. If desired, fix a smooth structure on any collection of connected components of $\partial M$. There is a smooth structure on $M$ extending the given smooth structure on (a subset of) $\partial M$. 
\end{theorem}

We are now ready for the construction. 

\begin{construction}[\citelist{\cite{gompf-infiniteexotic}*{Lemma~1.1}\cite{gompf-stipsicz-book}*{Exercise~9.4.23}}]\label{thm:knot-trace-exotic-R4}
Given a knot $K\subseteq S^3$ which is topologically slice but not smoothly slice there exists $\mathcal{R}$, a large exotic $\R^4$, with a smooth embedding $X_0(K)\hookrightarrow \mathcal{R}$. 
\end{construction}
 
\begin{proof}
Let $K$ be a knot which is topologically slice but not smoothly slice. Then there is a locally collared embedding $\varphi\colon X_0(K)\hookrightarrow \R^4$ by the trace embedding lemma. Then the image $\varphi(X_0(K))$ inherits a smooth structure from $X_0(K)$ since it is homeomorphic to the latter. Since $\varphi$ is locally collared, we see that $\R^4\sm \Int{\varphi(X_0(K))}$ is a manifold, and we can check using the Mayer--Vietoris sequence that it is connected. Since $X_0(K)$ is compact and $\R^4$ is not, we also know that $\R^4\sm \Int{\varphi(X_0(K))}$ is noncompact. So by \cref{thm:quinn-smoothing}, we extend the smooth structure on $\varphi(\partial X_0(K))$ to the rest of $\R^4\sm \Int{\varphi(X_0(K))}$. This produces a smooth structure on $\R^4$, and we denote the corresponding smooth $4$-manifold by $\mathcal{R}$. 

Suppose that $\mathcal{R}\diffeo \R^4_\std$. Then by construction we have a smooth embedding $X_0(K)\hookrightarrow \mathcal{R}\diffeo \R^4_\std$, so $K$ is smoothly slice by the trace embedding lemma, which is a contradiction. Therefore $\mathcal{R}$ is an exotic $\R^4$. 

Indeed, since $X_0(K)\subseteq \mathcal{R}$ has no smooth embedding in $\R^4_\std$, we see that $\mathcal{R}$ is large.\footnote{This holds for either definition of `large'.} 
\end{proof}

\subsection{Small exotic \texorpdfstring{$\R^4$s}{R4s} using ribbon disc exteriors}

The construction in this section will use \emph{Casson handles}. These are smooth noncompact $4$-manifolds, constructed by Casson in~\cite{casson-lectures}  as approximations of (open) $2$-handles. The boundary of a Casson handle $C$, called the \emph{attaching region} $\partial C$, is identified with the open solid torus $S^1\times \Int{D^2}$, and from the Kirby diagrams of Casson handles (see, e.g.~\citelist{\cite{kirbybook}*{Chapter~XII}\cite{gompf-stipsicz-book}*{Chapter~6}\cite{F}*{Section~2}}), one observes that every Casson handle $C$ admits a smooth embedding $(C,\partial C)\hookrightarrow (D^2\times D^2,S^1\times \Int{D^2})$, extending the aforementioned identification on the boundary. In~\cite{F}, Freedman showed that any Casson handle $C$ is homeomorphic, relative to its attaching region, to an open $2$-handle, i.e.~$(C,\partial C)\homeo(D^2\times \Int{D^2}, S^1\times \Int{D^2})$, again extending the aforementioned identification on the boundary. 

\begin{construction} Given a smoothly slice link $L$ with $n$ components, and a collection of Casson handles $\{C_1,\dots C_n\}$, there exists a small $\R^4$-homeomorph $\mathcal{R}_L$.
\end{construction}

\begin{proof}
Let $\{\Delta_i\}$ be a collection of smooth slice discs for $L$. Consider the complement $B^4\sm \bigcup \nu\Delta_i$, of open tubular neighbourhoods of the discs. If we glue in $2$-handles along the meridians of the components of $L$, we will get back $B^4$. However, we could instead glue in the Casson handles $\{C_1,\dots C_n\}$ along those meridians.  Let $\mathcal{R}_L$ denote the result of gluing in Casson handles to $B^4\sm \bigcup \nu\Delta_i$ along the meridians of $L$, and then removing all the remaining boundary. Since Casson handles are homeomorphic to open $2$-handles relative to the attaching region we see that $\mathcal{R}_L$ is an $\R^4$-homeomorph. Since every Casson handle embeds in a standard $2$-handle, respecting the attaching region, we also see that $\mathcal{R}_L$ admits a smooth embedding into $\R^4_\std$, and so by definition $\mathcal{R}_L$ is small.
\end{proof}

For certain choices of $L$ and Casson handles $\{C_1,\dots, C_n\}$, it can be shown that $\mathcal{R}_L$ is not diffeomorphic to $\R^4_\std$. For examples of this, see~\cites{demichelis-freedman,bizaca-gompf}. Roughly speaking, showing that such an $\mathcal{R}_L$ is exotic involves embedding it appropriately within a simply connected $h$-cobordism which is known to not be a smooth product. The known examples use ribbon links and are therefore called \emph{ribbon $\R^4$s}. All known small exotic $\R^4$s are ribbon $\R^4$s. The simplest known exotic ribbon $\R^4$ is built using the complement of a standard ribbon disc for the $9_{46}$ knot (also called the $(3,-3,3)$ pretzel knot), and the Casson handle built using self-plumbed $2$-handles within a single, positive self-plumbing at each stage~\cite{bizaca-gompf} -- this explicit construction yields a description of the corresponding exotic $\R^4$ as the interior of an infinite but rather simple handlebody.

\subsection{Universal exotic \texorpdfstring{$\R^4$}{R4}}  
Finally, there is the Freedman--Taylor universal exotic $\R^4$~\cite{freedman-taylor}, denoted by $\mathcal{U}$, which has the remarkable property that every $\R^4$-homeomorph $\mathcal{R}$ admits a smooth embedding $\mathcal{R}\hookrightarrow \mathcal{U}$. The construction of $\mathcal{U}$ also involves solving slicing problems for knots and links (via Kirby diagrams for Casson handles) and applying Quinn's result (\cref{thm:quinn-smoothing}). Unfortunately the construction is beyond the scope of the lectures and we invite the reader to learn more about the construction in~\cite{freedman-taylor} on their own.

\section{The knot concordance groups}\label{sec:conc-gps}

So far we have considered sliceness as a generalisation of triviality for knots. Similarly, there is a generalisation of isotopy, called \emph{concordance}. Specifically, we have the following definition. 

\begin{definition}
Knots $K,J\subseteq S^3$ are said to be smoothly (resp.\ topologically) \emph{concordant} if they cobound a smooth (resp.\ locally flat) embedded annulus in $S^3\times [0,1]$; see~\cref{fig:concordance-schematic}. More specifically, we consider $K\subseteq S^3$ as lying in $S^3\times \{1\}$ and $J\subseteq S^3$ as lying in $S^3\times \{0\}$, and assert the existence of a \emph{concordance} $A= S^1\times [0,1]\hookrightarrow S^3\times [0,1]$, with $S^1\times \{0\}$ mapping to $J$ and $S^1\times \{1\}$ mapping to $K$, where the embedding of $A$ is either smooth or locally flat, as needed.
\end{definition}

\begin{figure}[htb]
	\centering
\begin{tikzpicture}
        \node[anchor=south west,inner sep=0] at (0,0){	\includegraphics[width=8cm]{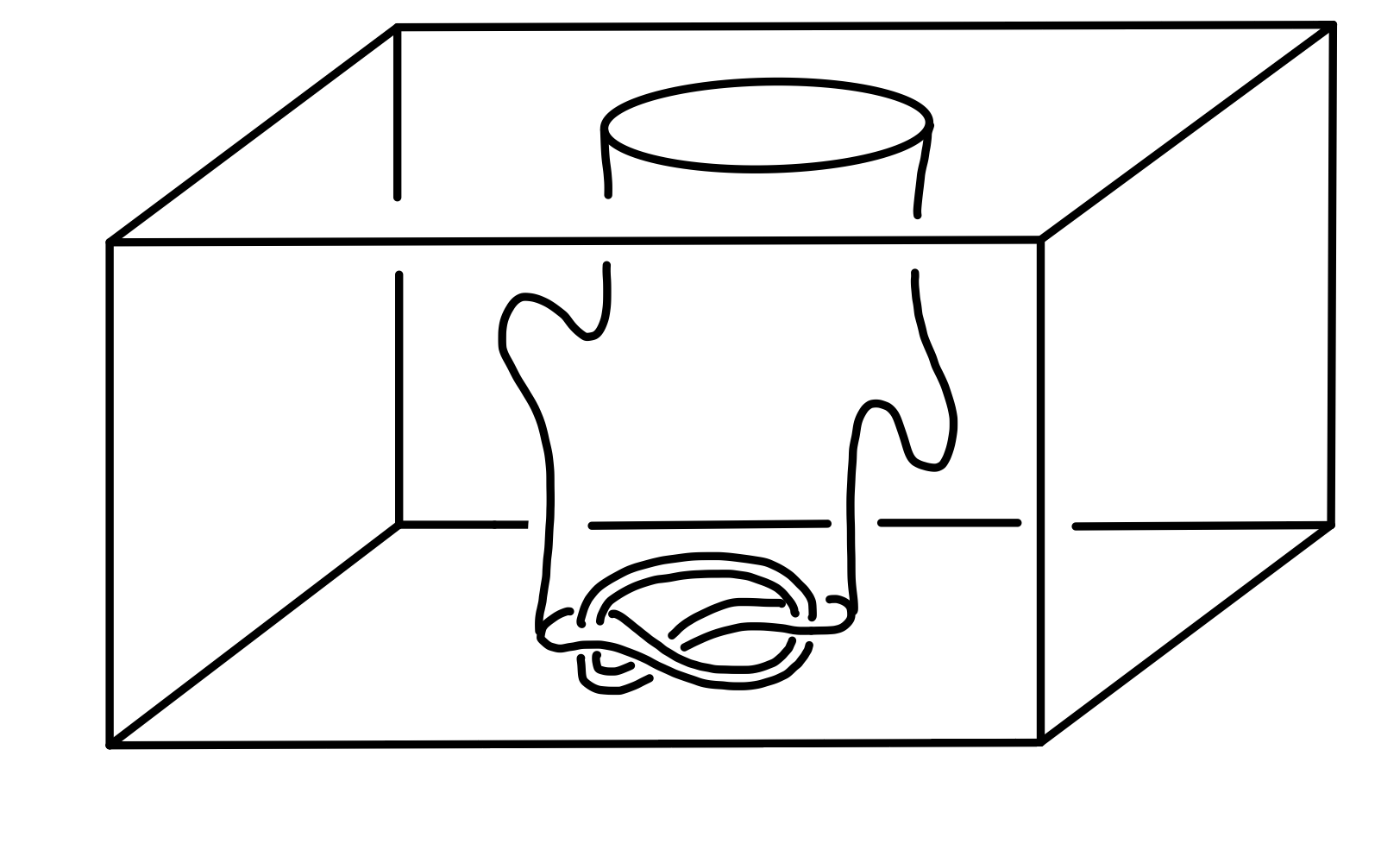}};
		%\draw[step=1cm,color=gray] (0,0) grid (8,5);
		\node at (8.75,3.25) {$S^3\times [0,1]$};
		\node at (5.15, 1.4) {$J$};
		\node at (5.65,4.2) {$K$};
	\end{tikzpicture}
\caption{A schematic picture showing a concordance between two knots $J\subseteq S^3\times \{0\}$ and $K\subseteq S^3\times \{1\}$.}\label{fig:concordance-schematic}
\end{figure}
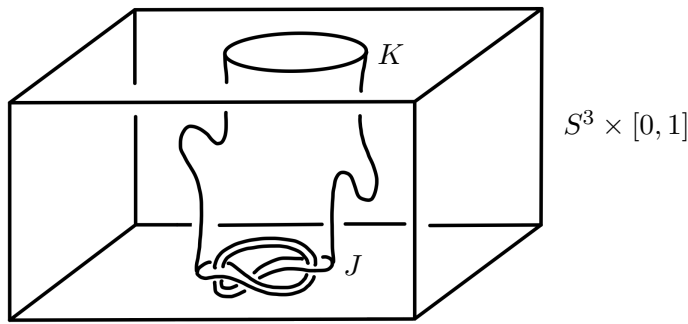 

\begin{proposition}[\green{Exercise~$\triangle$}~\ref{ex:slice-conc-unknot}]\label{prop:slice-conc-unknot}
A knot $K\subseteq S^3$ is smoothly (resp.\ topologically) slice if and only if it is smoothly (resp.\ topologically) concordant to the unknot.
\end{proposition}

Recall that by our conventions, all the manifolds in the definition of concordance are oriented. Using the standard orientation of $[0,1]$, we notice that $S^3\times \{1\}$ inherits the standard orientation of $S^3$ while $S^3\times \{0\}$ inherits the opposite orientation. Similarly a concordance (i.e.~a copy of $S^1\times [0,1]$) induces opposite orientations on its boundary components. Under these orientations, we see that if $K$ and $J$ are concordant knots, then the claimed annulus joins $K\subseteq S^3\times \{1\}=S^3$ to $rJ\subseteq S^3\times \{0\}=-S^3$. The latter knot is by definition $\ol{rJ}=r\ol{J}$. 

When speaking about slice knots we often ignore these orientation conventions, but they are key to proving the following fact. 

\begin{proposition}[\orange{Exercise~$\square$}~\ref{ex:conc-to-slice}]\label{prop:conc-to-slice}
Two knots $K,J\subseteq S^3$ are smoothly (resp.\ topologically) concordant if and only if $K\# r\ol{J}$ is smoothly (resp.\ topologically) slice.
\end{proposition}

Similar to isotopy, concordance (both smooth and topological) is an equivalence relation. However, unlike isotopy (\orange{Exercise~$\square$}~\ref{ex:iso-not-group}), the equivalence classes under concordance form a group under the connected sum operation.

\begin{proposition}[\green{Exercise~$\triangle$}~\ref{ex:conc-equiv-rel}]\label{prop:conc-equiv-rel}
Smooth (resp.\ topological) concordance is an equivalence relation on knots. 
\end{proposition}

\begin{definition}[Concordance groups]\label{def:conc-gp}
The \emph{smooth concordance group}, denoted by $\C$, consists of smooth concordance classes of knots, under the operation of connected sum. 

Similarly, the \emph{topological concordance group}, denoted by $\C^\top$, consists of topological concordance classes of knots, under the operation of connected sum.
\end{definition}

We often conflate a knot $K$ with its concordance class $[K]$. We asserted but did not prove that concordance classes form a group. This is the combination of the following proposition and \green{Exercise~$\triangle$}~\ref{ex:connsum-commutes}.

\begin{proposition}[\green{Exercise~$\triangle$}~\ref{ex:easy-conc-facts}]\label{prop:easy-conc-facts}
Let $K\subseteq S^3$ be a knot. In either $\C^\diff$ or $\C^\top$, the inverse of $[K]$ is the class of $-K:=r\ol{K}$. The identity element in $\C^\diff$ (resp.\ $\C^\top$) is the class of the unknot, equivalently the class of any smoothly (resp.\ topologically) slice knot. 
\end{proposition}

By \green{Exercise~$\triangle$}~\ref{ex:connsum-commutes}, the concordance groups are abelian, so it is tempting to speculate that they cannot be particularly complicated. This is somewhat of a red herring. There is an algebraic version of the knot concordance group, called the \emph{algebraic concordance group}, denoted by $\AC$. This group has a couple of equivalent definitions, one in terms of Seifert matrices and another in terms of Alexander modules and Blanchfield forms -- see~\cites{levine-invariants-knotcobordism,levine-knotcobordism,kearton-simple-BAMS,trotter-sequivalence,kearton-simple-TAMS,kearton-cobordism,stoltzfus,trotter-seifertmatrices} for further details. Given the definitions, it is not hard to see that there is a surjection
\[
\Phi\colon \C^\top \twoheadrightarrow \AC,
\]
given by taking a knot to the class of its Seifert matrix, or of its Alexander module and Blanchfield form. It was shown by J.~Levine~\cite{levine-knotcobordism} (see also Stoltzfus~\cite{stoltzfus}) that 
\[
\AC\cong \Z^\infty\oplus \Z/2^\infty \oplus \Z/4^\infty,
\]
so $\C^\top$ is quite large. 

Notice that $\AC$ has many elements of order two. So do $\C^\diff$ and $\C^\top$ by the following proposition.

\begin{proposition}[\green{Exercise~$\triangle$}~\ref{ex:neg-amph}]\label{prop:neg-amph}
A knot $K\subseteq S^3$ is said to be negative amphichiral if it is isotopic to $r\ol{K}$. Every negative amphichiral knot has order at most two in $\C^\diff$ and $\C^\top$. As a result, the figure eight knot has order two in $\C^\diff$ and $\C^\top$.
\end{proposition}

\begin{conjecture}[Open, Gordon~\cite{kirbylist}*{Problem~1.94}]
Let $K\subseteq S^3$ be a knot. If $K\# K$ is slice, then $K$ is smoothly concordant to a negative amphichiral knot.
\end{conjecture}

By contrast, there are no known candidates for knots of order four (note that $\AC$ has many such elements). We have the following open question.

\begin{question}[Open, see~\cite{kirbylist}*{Problem~1.32}]
Does the map $\Phi$ split? In particular, does $\C^\top$ (or $\C^\diff$) have any elements of order four? Indeed, are there any knots with finite order $\neq 0,2$ in $\C^\diff$ or $\C^\top$? 
\end{question}

The kernel of the map $\Phi$ has also been an object of much study, and is the focus of the next section. 

\begin{remark}\label{rem:highdimconcordance}
One may also consider high-dimensional knots, i.e.~embeddings $S^n\hookrightarrow S^{n+2}$ for arbitrary $n$. As in the classical case, one then has corresponding concordance and algebraic concordance groups. Indeed, for even $n$, every embedding $S^n\hookrightarrow S^{n+2}$ is slice, in the smooth, piecewise linear, or smooth categories~\citelist{\cite{kervaire-2knots}*{Th\'{e}or\`{e}me~III.6}\cite{kervaire-odd}*{Theorem~1}}. For odd $n\geq 5$, the topological and piecewise linear concordance groups of such embeddings is precisely in bijection with the corresponding algebraic concordance group (for $n=3$ there is a discrepancy of $\Z/2$ related to Rochlin's theorem)~\cites{levine-knotcobordism,kervaire-odd}. In the smooth category, the odd high-dimensional concordance groups are also closely determined by the corresponding algebraic concordance group, with a discrepancy of $\Z/k$ for some $k$ -- this discrepancy is related to the existence of exotic smooth structures on high-dimensional spheres and the famous Kervaire invariant problem (see~\cite{no-kervaire-one} for an account of the history of this problem). Roughly speaking, the map from knot concordance classes to the algebraic concordance group is still injective, as in the case for topological or piecewise linear concordance in high dimensions, but we no longer get a surjection since the construction in the realisation step may produce an exotic sphere rather than a standard one.
\end{remark}

From the definitions of the concordance groups it should be clear that there is a surjection 
\begin{equation}\label{eq:C-to-Ctop}
\Psi\colon \C^\diff \twoheadrightarrow \C^\top.
\end{equation}
As mentioned in~\cref{sec:defns}, there exist knots that are topologically slice but not smoothly slice. These (or rather, their smooth concordance classes) lie in the kernel of $\Psi$.

\begin{definition}
The subset $\T\subseteq \C^\diff$ consists of the smooth concordance classes of topologically slice knots. 
\end{definition} 

The structure of $\T$ is another key area of research (see e.g.~\cites{gompf-smoothconc,cochran-gompf,En95,hedden-kirk, HeLivRu12, CHH,Hom14,CH15,Hom15,HKL16,OzStipSz14,Hom19,cha-kim-bipolar,DHST21,KL22}). One method of studying this structure is described in the following section. 

\section{Filtrations of the knot concordance groups}\label{sec:filtrations}\label{sec:conc-gps-filtration}

\subsection{Overview}\label{sec:filtrations-overview}
As mentioned in \cref{rem:highdimconcordance}, the high-dimensional codimension two concordance groups of odd-dimensional knots are controlled almost entirely by the corresponding algebraic analogues, and in almost all cases they are isomorphic. Therefore, it was an interesting question for several years whether the knot concordance groups $\C^\top$ and the algebraic knot concordance group $\AC$ are isomorphic also in the classical dimension, in other words, whether the map $\Phi$ is a bijection. Since it is easily shown to be a surjection, the main question was whether $\Phi$ is injective. The first nontrivial elements in the kernel of $\Phi$ were given by Casson and Gordon in~\cites{CG78,CG86}. Then Cochran, Orr, and Teichner~\cite{COT1} defined the \emph{solvable filtration} of $\C^\diff$, denoted by $\{\F_{i}\}_{i\in \tfrac{1}{2}\NN}$, satisfying 
\[
\bigcap_{i\in \tfrac{1}{2}\NN} \F_i\subseteq \cdots \subseteq \F_{n+1}\subseteq \F_{n.5}\subseteq \F_n \subseteq \cdots \subseteq \F_0\subseteq \C^\diff. 
\]

\begin{remark}
Cochran--Orr--Teichner defined the filtration in the locally flat category, for the topological concordance group $\C^\top$. In light of future developments we give the definition in the smooth category. For the relationship between the two notions see~\cref{prop:smoothvtop-solvable,cor:smoothvtop-quotients-isomorphic}.
\end{remark}

We will give the precise definition of the terms of this filtration presently. For now, we note that it is highly nontrivial (see \cref{thm:solvable-filt-nontrivial}), and it subsumes several other invariants, e.g.~a knot $K$ is in $\F_0$ if and only if $\Arf(K)=0$; it is in $\F_{0.5}$ if and only if it is algebraically slice; and if $K$ is in $\F_{1.5}$ then all its Casson--Gordon sliceness obstructions vanish (see~\red{Exercise~$\bigcirc$}~\ref{ex:arf-0solvable}). We will also see in \cref{thm:T-in-solvable-int} that $\T\subseteq \bigcap_{i\in \tfrac{1}{2}\NN} \F_i$.

\begin{question}[Open]
Is $\T = \bigcap_{i\in \tfrac{1}{2}\NN} \F_i$?
\end{question}

On the other hand, as we mentioned in \cref{sec:defns}, we know that $\T\neq \{1\}$. By \cref{thm:T-in-solvable-int} the solvable filtration cannot effectively distinguish between smooth concordance classes of topologically slice knots. For this purpose, Cochran--Harvey--Horn~\cite{CHH} defined the \emph{bipolar filtration} of $\T$, denoted by $\{\T_{n}\}_{n\in \NN}$, satisfying 
\[
\bigcap_{i\in\NN} \T_i\subseteq \cdots \subseteq \T_{n+1}\subseteq \T_n \subseteq \cdots \subseteq \T_0\subseteq \T. 
\]
We will define the bipolar filtration presently, and note here that it is also highly nontrivial (see \cref{thm:bipolar-filt-nontrivial}).

\begin{question}[Open]
Is $\bigcap_{i\in \NN} \T_i=\{1\}$?
\end{question}

\subsection{Definitions of the filtrations}
Both the solvable filtration of Cochran--Orr--Teichner and the bipolar filtration of Cochran--Harvey--Horn are motivated by the following characterisation of topologically slice knots. For a knot $K\subseteq S^3$, the $3$-manifold obtained by performing $0$-framed Dehn surgery on $S^3$ along $K$ is denoted by $S^3_0(K)$. 

\begin{proposition}\label{prop:0surgerychar}
Let $K\subseteq S^3$ be a knot. Then $K$ is topologically slice if and only if $S^3_0(K)=\partial W$, where $W$ is a compact, connected, oriented $4$-manifold such that 
\begin{enumerate}[(i)]
\item\label{item:H1} inclusion induces an isomorphism $\Z\cong H_1(S^3_0(K);\Z)\to H_1(W;\Z)$; 
\item\label{item:H2} $H_2(W;\Z)=0$; and
\item\label{item:pi1} $\pi_1(W)$ is normally generated by the meridian $\mu_K\subseteq S^3_0(K)$.
\end{enumerate} 
\end{proposition}

\begin{proof}
Assume that $K$ is topologically slice, with a topological slice disc $\Delta\subseteq B^4$. As mentioned before, by work of Quinn~\citelist{\cite{quinn-endsiii}*{Theorem~2.5.1}\cite{FQ}*{Theorem~9.3}} we know that $\Delta$ admits an open tubular neighbourhood $\nu\Delta\homeo \Delta\times \mathring{D}^2$. Then let $W= B^4\sm \nu\Delta$. The boundary of $W$ is a union of the knot exterior $S^3\sm \nu K$ and $\Delta\times S^1$, a solid torus. This shows that $\partial W$ is the result of some Dehn surgery on $S^3$ along $K$. In order to see that we indeed get $S^3_0(K)$ as the boundary, note that the disc $\Delta\times \mathrm{pt}\subseteq \Delta\times S^1$ is attached to a pushoff of $K$ which has zero linking number with $K$, since in particular it bounds a disc in $B^4$ disjoint from $\Delta\times \{0\}\subseteq \nu\Delta$.

Then \ref{item:pi1} follows from the Seifert--van Kampen theorem since we obtain $B^4$ (which is simply connected) from $W$ by gluing in a thickened $2$-cell (i.e.~the neighbourhood $\nu\Delta$) to $W$ along $\mu_K\subseteq \partial W$. 

To see \ref{item:H2}, consider the Mayer--Vietoris sequence for $B^4$ as the union $W\cup \nu \Delta=B^4$. Notice that $W\cap \nu \Delta\cong D^2\times S^1$. Then we have 
\[
H_2(D^2\times S^1;\Z)\to H_2(W;\Z)\oplus H_2(\nu \Delta;\Z)\to H_2(B^4;\Z),
\]
so we see that $H_2(W;\Z)=0$. From the same sequence, we also get that $H^2(W;\Z)\cong H^3(W;\Z)=0$. Therefore, $H_1(W,\partial W;\Z)\cong H_2(W,\partial W;\Z)=0$. Apply these to the long exact sequence for the pair $(W,\partial W)$, where we know that $\partial W\cong S^3_0(K)$:
\[
H_2(W,S^3_0(K);\Z)\to H_1(S^3_0(K);\Z)\to H_1(W;\Z)\to H_1(W,S^3_0(K);\Z)
\]
This yields \ref{item:H1} and completes half of the proof. 

For the converse direction, given $W$, we will glue in a $2$-handle along $\mu_K\subseteq \partial W$, with a framing we describe presently, so that the result $\B$ is homeomorphic to $B^4$, within which we will locate a slice disc for the knot $K\subseteq \partial \B$.  To do so, we will need to choose a framing of $\mu_K$ with the goal of recovering $S^3$ as $\partial \B$.  We do this explicitly next. Experts can safely skip this level of detail, and instead work directly in a Dehn surgery diagram. We provide an explicit argument for the benefit of newcomers to the field.
 
This construction will use multiple copies of the $2$-disc $D^2$, which we indicate with subscripts, i.e. each $D^2_i$ below indicates a copy of $D^2$. Use the Seifert longitude of $K$ to identify a tubular neighbourhood $\nu K\subseteq S^3$ with $K\times D^2_0$, so the Seifert longitude is given by $K\times x$ for some $x\in \partial D^2_0$. By definition, we know that 
\[
S^3_0(K)= (S^3\sm (K\times \mathring{D_0^2})) \cup_{\partial D^2_1\times S^1} (D^2_1\times S^1),
\]
with the \emph{surgery solid torus} $D^2_1\times S^1$ attached to $K\times \partial D^2_0$ so that $\partial D^2_1\times y$ for some $y\in S^1$ is identified with $K\times x$ and $z\times S^1$ for $z\in\partial D^2_1$ is identified with the meridian $\mu_K$. Attach a $2$-handle $D^2_1\times D^2_2$ to $W$ along the core of the surgery solid torus, i.e.~along $0\times S^1$, using the surgery solid torus as the requisite framing of a tubular neighbourhood of the core, and call the result $\B$. In other words we have: 
\[
\B:= W\cup_{D^2_1\times S^1} D^2_1\times D^2_2,
\]
Note that the core $0\times S^1\subseteq D^2_1\times S^1$ is isotopic in $S^3_0(K)$ to $z\times S^1$, which we saw is identified with $\mu_K$, so we are indeed attaching a $2$-handle to $W$ along $\mu_K$ as claimed. 

Then $\partial \B$ is the result of surgery on $S^3_0(K)$ along $\mu(K)$. More precisely, we have
\begin{align*}
\partial \B=& \bigl(S^3_0(K)\sm (D^2_1\times S^1)\bigr) \cup_{S^1\times \partial D^2_2} (S^1\times D^2_2)\\
=&\biggl(\Bigl(\bigl(S^3\sm (K\times \mathring{D_0^2})\bigr)\cup_{\partial D^2_1\times S^1} (D^2_1\times S^1)\Bigr) \sm (D^2_1\times S^1)\biggr) \cup_{S^1\times \partial D^2_2} (S^1\times D^2_2)\\
=&\bigl(S^3\sm (K\times D^2_0)\bigr) \cup_{S^1\times \partial D^2_2} (S^1\times D^2_2),
\end{align*}
where we can check that $w\times \partial D^2_2$ for $w\in S^1$ is attached to the meridian $\mu_K$. In other words, we have performed the $\infty$-framed Dehn surgery on $S^3$ along $K$, producing $\partial \B=S^3$. In this copy of $S^3$, we still have the knot $K$, represented as $\partial D^2_1 \times y$, so in particular the disc $D^2_1\subseteq \B$ is a slice disc for $K$. 

Moreover, by our hypotheses, we also see that $\B$ is homotopy equivalent to $B^4$. Then by the topological $4$-dimensional Poincar\'e conjecture~\cite{F}*{Theorem~1.6}, we know that $\mathcal{B}\homeo B^4$. This completes the proof. 
\end{proof}

We now give the definition of the solvable filtration. Recall that for a group $G$,  the \emph{$i$th derived subgroup}, denoted $G^{(i)}$, is inductively defined by setting $G^{(0)}=G$ and $G^{(i+1)}:=[G^{(i)}, G^{(i)}]$ for each $i\geq 0$.

\begin{definition}[\cite{COT1}*{Definition~1.2}]\label{def:solvable-filtration}
Let $K\subseteq S^3$ be a knot and $n\in \NN$. We say that $K$ is (smoothly) \emph{$n$-solvable} if $S^3_0(K)=\partial W$ where $W$ is a smooth, compact, connected, oriented $4$-manifold such that 
\begin{enumerate}[(i)]
\item the inclusion induces an isomorphism $\Z\cong H_1(S^3_0(K);\Z)\to H_1(W;\Z)$; and
\item $H_2(W;\Z)$ has a basis consisting of smoothly embedded, closed, connected, oriented surfaces $\{L_i,D_i\}_{i=1}^k$, for some $k$, such that 
\begin{enumerate}
\item each $L_i$ and $D_i$ has trivial normal bundle, and 
\item the surfaces $\{L_i,D_i\}_i$ are pairwise disjoint, except that for each $i$, the surface $L_i$ intersects $D_i$ transversely once with positive sign; and
\item for each $i$, we have $\pi_1(L_i)\subseteq \pi_1(W)^{(n)}$ and $\pi_1(D_i)\subseteq \pi_1(W)^{(n)}$, with respect to the inclusion induced maps. 
\end{enumerate}
\end{enumerate}
The set of $n$-solvable knots is denoted by $\F_n$ and the manifold $W$ is called an \emph{$n$-solution}.

We say that $K$ is (smoothly) \emph{$n.5$-solvable} if, in addition, for each $i$, we have  $\pi_1(D_i)\subseteq \pi_1(W)^{(n+1)}$ with respect to the inclusion induced maps. The set of $n.5$-solvable knots is denoted by $\F_{n.5}$ and the manifold $W$ is called an \emph{$n.5$-solution}.
\end{definition}

We leave it to the reader to show that, for each $n\in\tfrac{1}{2}\NN$, the set $\F_n$ is a subgroup of $\C^\diff$ (\green{Exercise~$\triangle$}~\ref{ex:solvable-subgroup}). Note that condition (i) of \cref{def:solvable-filtration} and of \cref{prop:0surgerychar} are the same. Condition (ii) in \cref{def:solvable-filtration} can be seen as a generalisation of condition (ii) in \cref{prop:0surgerychar} -- roughly speaking, if we assume that each surface $L_i$ is a sphere, then one could do surgery on $W$ along $\{L_i\}$, i.e.~remove a tubular neighbourhood $L_i\times D^2$ and glue in $S^1\times D^3$, for each $i$, and the resulting $4$-manifold with boundary $S^3_0(K)$ would have trivial second homology. The condition that $\pi_1(L_i)\subseteq \pi_1(W)^{(n)}$ then measures how far away the homology class of $L_i$ is from being represented by an immersed sphere: in the case that $\pi_1(L_i)$ is mapped to the trivial subgroup of $\pi_1(W)$, one could perform ambient surgery along half a symplectic basis of curves on $L_i$ to transform it to an immersed sphere.

\begin{remark}
We could also add in Condition (iii) from \cref{prop:0surgerychar} to \cref{def:solvable-filtration}, without any loss to known results about the solvable filtration. \textit{A priori} (but only conjecturally) that would lead to a different filtration.  
\end{remark}

In \cref{def:solvable-filtration}, if the manifold $W$ is not required to be smooth and the surfaces $\{L_i,D_i\}_i$ are only required to be locally flat embedded, we say that $K$ is \emph{topologically} $n$-solvable or $n.5$-solvable. Denote the set of such knots by $\F^\top_n$ and $\F^\top_{n.5}$ respectively. \textit{A priori} the filtrations $\{\F^\top_n\}_{n\in\frac{1}{2}\NN}$ and $\{\F_n\}_{n\in\frac{1}{2}\NN}$ appear to be distinct, but indeed, they coincide, as mentioned in~\citelist{\cite{CHL1}*{p.\ 1423}\cite{CHH}*{p.\ 2105}}. Since an explicit proof has not appeared in the literature before, we provide one next. 

\begin{proposition}\label{prop:smoothvtop-solvable}
Let $K\subseteq S^3$ be a knot and $n\in \tfrac{1}{2}\NN$. Then $K$ is smoothly $n$-solvable if and only if it is topologically $n$-solvable. 
\end{proposition}

\begin{proof}
By forgetting smoothness, we see that if $K\in \F_n$ then $K\in \F^\top_n$. 

Assume that $K\in \F^\top_n$ and let $W$ denote a topological $n$-solution. First we note that the Kirby--Siebenmann invariant $\ks(W)=0$. To see this, observe that $W$ is spin, since $H_1(W;\Z)\cong \Z$ has no $2$-torsion and the intersection form is even. The Kirby--Siebenmann invariant of a compact, oriented, spin topological $4$-manifold can be computed using the signature and the Rochlin invariant of the boundary with respect to the induced spin structure~\cite{FQ}*{Proposition~10.2B}. To avoid the Rochlin invariant and having to specify spin structures, we instead glue the $0$-trace of $K$, denoted by $X_K$, to $W$, along the common boundary $S^3_0(K)$; call the result $Y$. Since $X_K$ is smooth, we know that $\ks(Y)=\ks(W)$. Using the Mayer--Vietoris sequence, 
\[
\begin{tikzcd}
H_1(S^3_0(K);\Z)\arrow[r, "\cong"]	&H_1(X_K;\Z)\oplus H_1(W;\Z)\arrow[r, two heads]	&H_1(Y;\Z)\arrow[r]	&0,
\end{tikzcd}
\]
where we used that $H_1(X_K;\Z)=0$, we see that $H_1(Y;\Z)=0$. Using the Mayer--Vietoris sequence again, noting that $H_3(Y;\Z)\cong H^1(Y;\Z)=0$, and that the inclusion induced map $H_2(S^3_0(K);\Z)\to H_2(X_K;\Z)$ is an isomorphism while $H_2(S^3_0(K);\Z)\to H_2(W;\Z)$ is trivial, we can further conclude that $H_2(Y;\Z)\cong H_2(W;\Z)$ and indeed that the intersection forms on $Y$ and on $W$ coincide. So we have that $Y$ is spin, since it has even intersection form and trivial first homology. By the definition of an $n$-solution, the signature $\sigma(Y)=\sigma(W)=0$. For this we can either use that the intersection form of $Y$ agrees with that of $W$, or use Novikov additivity. Putting everything together, we see that $\ks(W)=\ks(Y)=0$ by~\cite{FQ}*{Proposition~10.2B}. Therefore, the connected sum of $W$ with sufficiently many copies of $S^2\times S^2$ is smoothable~\cite{FQ}*{Theorem~8.6} (see also~\citelist{\cite{4dguide}*{Theorem~8.6}\cite{DETbook-flowchart}*{Section~21.4.5}}). Choose $r\geq 0$ so that $W'=W\# rS^2\times S^2$ is smoothable. Choose a smooth structure on $W'$. We will show that $W'$ is a (smooth) $n$-solution. 

As required, the manifold $W'$ is smooth, compact, connected and oriented. We also know that inclusion induces an isomorphism $H_1(W;\Z)\cong H_1(W';\Z)$. So it remains only to check the conditions on the second homology. Let $\{L_i,D_i\}_{i=1}^k$ denote a basis for $H_2(W;\Z)$ provided by the definition of an $n$-solution. By requiring that the connected sum operation is performed away from $\{L_i,D_i\}_i$, we may assume that each $L_i$ and $D_i$ also lies in $W'$. Let $\{A_i,B_i\}_{i=1}^r$ denote the standard basis for $H_2(rS^2\times S^2;\Z)$ given by the $S^2$-factors, oriented so that each $A_i$ intersects the corresponding $B_i$ with positive sign. Of course each $A_i$ and $B_i$ has trivial normal bundle. We also know that $\pi_1(W')\cong \pi_1(W)$. Therefore, the set $\{L_i,D_i\}_{i=1}^k\cup \{A_i,B_i\}_{i=1}^r$ is a basis for $H_2(W';\Z)$ satisfying all the conditions of \cref{def:solvable-filtration}, except that the elements of $\{L_i,D_i\}_i$ may not be smoothly embedded in $W'$. 

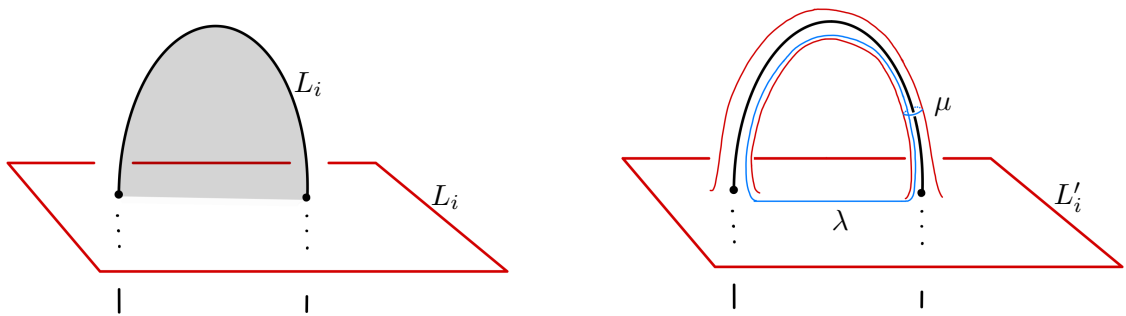
\begin{figure}[htb]
	\centering
\begin{tikzpicture}
        \node[anchor=south west,inner sep=0] at (0,0){	\includegraphics[width=15cm]{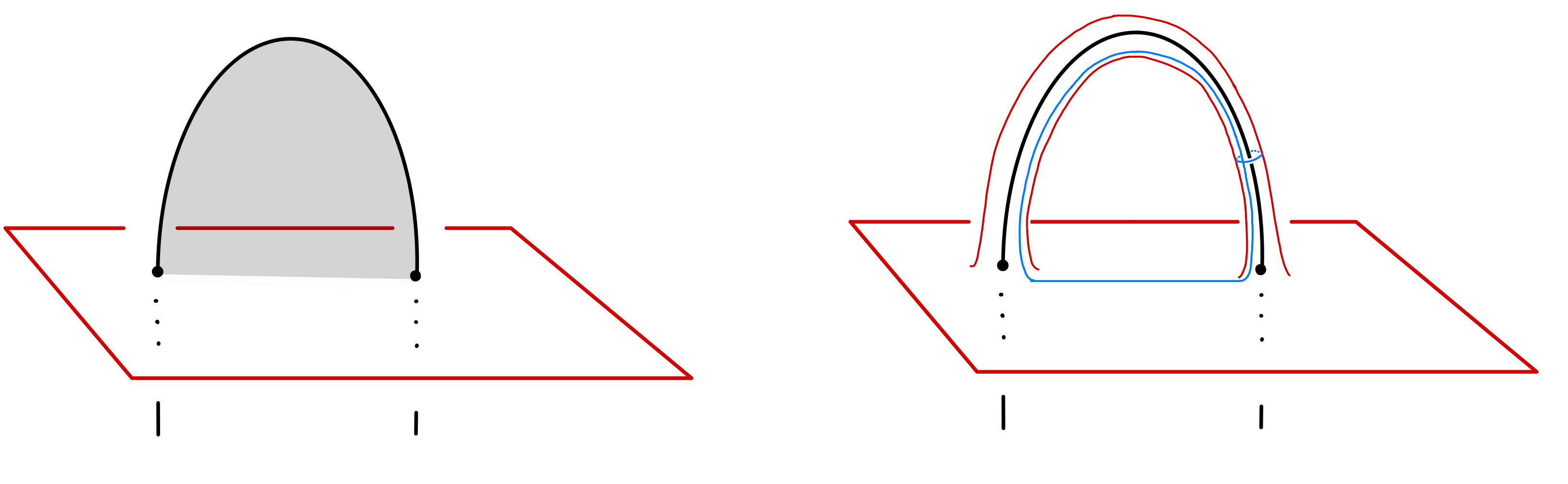}};
%\draw[step=1cm,color=gray] (0,0) grid (15,5);
		\node at (5.8,2) {$L_i$};
		\node at (4,3.5) {$L_i$};
		\node at (14,2) {$L_i'$};
		\node at (12.35,3.2) {$\mu$};
		\node at (11,1.7) {$\lambda$};				
	\end{tikzpicture}
\caption{Left: The result of performing a finger move on $L_i$. The Whitney disc is shaded in grey. Right: The surface $L_i'$ is produced after tubing along a Whitney arc. The blue curves show the new generators of $\pi_1(L_i')$, denoted by $\mu$ and $\lambda$, compared to $\pi_1(L_i)$.} \label{fig:whitneyarc-tube}
\end{figure} 

In general, it may not be possible to change $\{L_i,D_i\}_i$ by a smooth ambient isotopy so that the result consists of smooth embeddings. However, by~\cite{FQ}*{Theorem~8.1A} (see also~\cite{FQ}*{p.~115}), we may modify each $L_i$ and $D_i$ by ambient isotopies and finger moves until the resulting immersions are smooth. The new intersections are paired by framed, embedded Whitney discs. Tube each $L_i$ (or $D_i$) to itself, along one of each pair of Whitney arcs, as shown in~\cref{fig:whitneyarc-tube}; call the resulting surfaces $\{L_i',D_i'\}_i$. Then by construction each $L'_i$ and $D'_i$ is smoothly embedded in $W'$ as needed. We also note that the intersections are as before, since the modification has occurred along arcs which can be assumed to have interiors disjoint from $\{L_i,D_i\}$ and one another. Of course, the genera of the surfaces have increased so we have to ensure that the new elements of the fundamental group still map to $\pi_1(W)^{(n)}$ or $\pi_1(W)^{(n+1)}$, as appropriate. We consider the case of a surface $L_i'$ -- the case of a $D_i'$ is similar. As shown in~\cref{fig:whitneyarc-tube}, for each finger move, there are two new generators of the $\pi_1(L_i')$ compared to $\pi_1(L_i)$, given by the curves $\mu$ and $\lambda$. However, both $\mu$ and $\lambda$ are null-homotopic in $W$, the latter due to the existence of the Whitney disc and the former due to a meridional disc for $L_i'$. Therefore, the images of $\pi_1(L_i)$ and $\pi_1(L_i')$ in $\pi_1(W)$ under the inclusion induced maps coincide. This completes the proof that $W'$ is a smooth $n$-solution. 
\end{proof}

Consider the natural map $\Psi\colon\F_n\to\F^\top_n$, where for each knot $K$ the smooth concordance class of $K$ is sent to the topological concordance class, i.e.~induced by the map $\Psi$ in~\eqref{eq:C-to-Ctop}. Then \cref{prop:smoothvtop-solvable} has the following straightforward corollary.

\begin{corollary}\label{cor:smoothvtop-quotients-isomorphic}
For each $n\in\tfrac{1}{2}\NN$, the map $\Psi$ induces an isomorphism $\F_n/\F_{n+0.5}\cong \F_n^\top/\F^\top_{n+0.5}$.
\end{corollary}

We now state the results mentioned in \cref{sec:filtrations-overview}, as well as a few more open questions.

\begin{theorem}[\cites{COT1,COT2,CT07,CHL1,CHL2}]\label{thm:solvable-filt-nontrivial}
For each $n\in \NN$, there is a subgroup of $\F_n/\F_{n.5}$ isomorphic to $\Z^\infty\oplus(\Z/2)^\infty$. 
\end{theorem}

For other related work, see also~\cites{jiang81,livingston-order2,DPR-cabling}. By contrast, little is known about the `other half' of the filtration. 

\begin{question}[Open]
For some $n\in\NN$, is the group $\F_{n.5}/\F_{n+1}$ nontrivial? 
\end{question}

For example, it was shown in~\cite{DMOP} that genus one knots which are $0.5$-solvable (equivalently, algebraically slice) are also $1$-solvable. We note that there is an analogue of the solvable filtration for $m$-component (string) links, denoted by $\{\F_n^m\}_{n\in \tfrac{1}{2}\NN}$, where for large enough $m$ it is known that $\F^m_{n.5}/\F^m_{n+1}$ is nontrivial~\cite{Otto-grope}.

Examples of knots lying in $\F_n$ for large $n$ are constructed using \emph{satellite operations}. We will describe this further in~\cref{sec:satellites}. For now, we note that the examples from~\cites{COT1,COT2,CHL1,CHL2} all have genus one. The examples from~\cite{DPR-cabling} are cables of genus one knots, and have large Seifert genus and smooth slice genus, but may well have topological slice genus one. In general it is difficult to bound the topological slice genus of knots lying deep in the solvable filtration. 

\begin{question}[Open, Cha~\cite{cha-minimal}*{Remark~5.6}]\label{question:solvable-genus}
For arbitrary $n>2$ and $g>1$, does there exist a knot in $\F_n$ with topological slice genus at least $g$? 
\end{question}

The $n=0$ and $1$ cases can be shown using Levine--Tristram and Casson--Gordon signatures respectively, while the $n=2$ case was shown by Cha--Miller--Powell in~\cite{cha-miller-powell}.

Membership in the various levels of the solvable filtration can be obstructed using \emph{von Neumann $\rho$-invariants}, defined by Cheeger and Gromov in~\cite{cheeger-gromov}. For a detailed discussion, see \citelist{\cite{CHL-fractal}*{Section 5}\cite{CT07}*{Section 2}\cite{COT1}*{Section 2}}.

\begin{theorem}\label{thm:T-in-solvable-int}
Let $K\subseteq S^3$ be a topologically slice knot. Then $K$ lies in $\F_n$ for every $n$. In other words, $\T\subseteq \bigcap_{n\in \tfrac{1}{2}\NN}\F_n$. 
\end{theorem}

\begin{proof}
Since $K$ is topologically slice, by \cref{prop:0surgerychar} we can see that $K\in \F_n^\top$ for every $n$ by definition. Then by \cref{prop:smoothvtop-solvable}, we also know that $K\in \F_n$ for every $n$.
\end{proof}

The above theorem implies that the solvable filtration is not useful for studying the structure of $\T$. However, there is another filtration which can be used for this purpose. 

\begin{definition}[\cite{CHH}*{Definition~5.1}]\label{def:bipolar-filtration}
Let $K\subseteq S^3$ be a knot and $n\in \NN$. We say that $K$ is \emph{$n$-positive} if $S^3_0(K)=\partial W$ where $W$ is a smooth, compact, connected, oriented $4$-manifold such that 
\begin{enumerate}[(i)]
\item the inclusion induces an isomorphism $\Z\cong H_1(S^3_0(K);\Z)\to H_1(W;\Z)$; 
\item $H_2(W;\Z)$ has a basis consisting of smoothly embedded, disjoint, closed, connected, oriented surfaces $\{S_i\}_{i=1}^k$, for some $k$, such that 
\begin{enumerate}
\item each $S_i$ has normal bundle with Euler number $+1$; and 
\item for each $i$, we have $\pi_1(S_i)\subseteq \pi_1(W)^{(n)}$ with respect to the inclusion induced maps; and
\end{enumerate}
\item $\pi_1(W)$ is normally generated by the meridian $\mu_K\subseteq S^3_0(K)$.
\end{enumerate}
The set of $n$-positive knots is denoted by $\P_n$ and the manifold $W$ is called an \emph{$n$-positon}.

We say that $K$ is \emph{$n$-negative} if each surface $S_i$ instead has normal bundle with Euler number $-1$. The set of $n$-negative knots is denoted by $\N_{n}$ and the manifold $W$ is called an \emph{$n$-negaton}.

Let $\B_n:= \P_n\cap \N_n$. Knots in $\B_n$ are said to be \emph{$n$-bipolar}. 
\end{definition}

We leave it to the reader to show that, for each $n\in\NN$, the sets $\P_n$ and $\N_n$ are submonoids of $\C^\diff$, and $\B_n$ is a subgroup of $\C^\diff$ (\green{Exercise~$\triangle$}~\ref{ex:bipolar-subgroup}). As for the solvable filtration, note that condition (i) of \cref{def:bipolar-filtration} and of \cref{prop:0surgerychar} are the same. Indeed, now we further have the same condition (iii). Condition (ii) in \cref{def:bipolar-filtration} is again a generalisation of condition (ii) in \cref{prop:0surgerychar} -- roughly speaking, if we assume that each surface $S_i$ is a sphere, then one could perform a blow down operation on $W$ along $\{S_i\}$, i.e.~remove a neighbourhood of $S_i$ diffeomorphic to the $D^2$-bundle over $S^2$ with Euler number $\pm 1$, and then glue in $B^4$, for each $i$, and the resulting $4$-manifold with boundary $S^3_0(K)$ would have trivial second homology. The condition that $\pi_1(S_i)\subseteq \pi_1(W)^{(n)}$ as before measures how far away the homology class of $S_i$ is from being represented by an immersed sphere: in the case that $\pi_1(S_i)$ is mapped to the trivial subgroup of $\pi_1(W)$, one could perform ambient surgery along half a symplectic basis of curves on $S_i$ to transform it to an immersed sphere.

\begin{definition}
For each $n\in \NN$ define $\T_n:=\T\cap \B_n$. The corresponding filtration $\{\T_n\}_{n\in\NN}$ is called the \emph{bipolar filtration} of $\T$. 
\end{definition}

As desired, the bipolar filtration is highly nontrivial on $\T$, as seen in the theorem below.

\begin{theorem}[\cite{cha-kim-bipolar}, see also~\cites{CHH,CH15}]\label{thm:bipolar-filt-nontrivial}
For each $n$, there is a subgroup of $\T_n/\T_{n+1}$ isomorphic to $\Z^\infty$.
\end{theorem}

As before, the examples of knots lying in $\T_n$ for large $n$ used in the theorem above are constructed via a generalisation of the satellite operation, which we will describe in~\cref{sec:satellites}. As in the case of the solvable filtration. membership in the various levels of the positive, negative, and bipolar filtrations of $\C^\diff$ can be obstructed using the von Neumann $\rho$-invariants. However, these are not sufficient to prove nontriviality of the bipolar filtration of $\T$. The proof of \cref{thm:bipolar-filt-nontrivial} combines von Neumann $\rho$-invariants with the Heegaard--Floer $d$-invariant. For more details, see~\cites{CHH,CH15,cha-kim-bipolar}.

There are several open questions about the bipolar filtration, such as the following. 

\begin{question}[Open]
Is there a (short) list of invariants characterising knots lying in $\B_0$? Note that there are explicit characterisations of knots in $\F_0$ and $\F_{0.5}$ (\red{Exercise~$\bigcirc$}~\ref{ex:arf-0solvable}).
\end{question}

\begin{question}[Open]
For $n\in \NN$, is there a subgroup of $\T_n/\T_{n+1}$ isomorphic to $(\Z/2)^\infty$? (See \cref{thm:solvable-filt-nontrivial,thm:bipolar-filt-nontrivial}.)
\end{question}

\begin{question}[Open]
For arbitrary $n\geq 0$ and $g>1$, does there exist a knot in $\T_n$ with smooth slice genus at least $g$? (See \cref{question:solvable-genus}.)
\end{question}
Unlike \cref{question:solvable-genus}, the above appears to be open even in the case $n=0$. 

\section{Satellite operations on knots}\label{sec:satellites}

In this final section, we discuss satellite operations on knots. Recall that Whitehead doubling is a special case of the satellite operation, and satellite operations can be used to construct examples of knots deep in the solvable and bipolar filtrations. 

\begin{definition}\label{def:satellite} A \emph{pattern} $P$ is a knot in the solid torus~$S^1\times D^2$. Given a knot $K$, called the \emph{companion}, the \emph{satellite knot} $P(K)$ is obtained by tying the solid torus into the knot $K$, in an appropriately untwisted manner, as shown in \cref{fig:whitehead-doubling}.  More precisely, we map $S^1\times D^2$ to a tubular neighbourhood $\nu K$ of $K$, so that $S^1\times x$ for $x\in \partial D^2$ is mapped to the Seifert longitude of $K$. The image of $P$ under this map is the satellite knot $P(K)$ by definition. Note that the knot $P(K)$ inherits an orientation from $P$ and $K$. 

The algebraic intersection number of $P$ with a generic meridional disc of $S^1\times D^2$ is called the \emph{winding number} of~$P$. For example, the winding number of the Whitehead doubling patterns $\Wh^\pm$ is zero.
\end{definition}

The satellite operation is compatible with concordance, and yields well defined functions on the concordance groups, called \emph{satellite operators} (see~\green{Exercise~$\triangle$}~\ref{ex:satellite-concordance}). 

We next give an alternative view of the satellite construction. In this case, we begin with a knot $R\subseteq S^3$ and a curve $\eta\subseteq S^3\sm R$ so that $\eta$ is unknotted when considered in $S^3\supseteq S^3\sm R$. Then note that the complement in $S^3$ of a tubular neighbourhood $\nu\eta$ of $\eta$ can be canonically identified with the solid torus, so $R\subseteq S^3\sm \nu \eta$ is a pattern. So, given any arbitrary knot $K$, we could construct the satellite knot with respect to this pattern. In this case, we denote the result by $R_\eta(K)$, and call it the \emph{result of infection on $R$ by $K$ along $\eta$}. This procedure is depicted in~\cref{fig:infection}. There is a generalisation of the infection procedure where as input we can take not just knots but arbitrary \emph{string links}. For more on this procedure, see~\citelist{\cite{cochran-orr-homology}*{Section~1}\cite{cochran-nckt}*{p.~385}\cite{CFT09}*{Section~2.2}}.

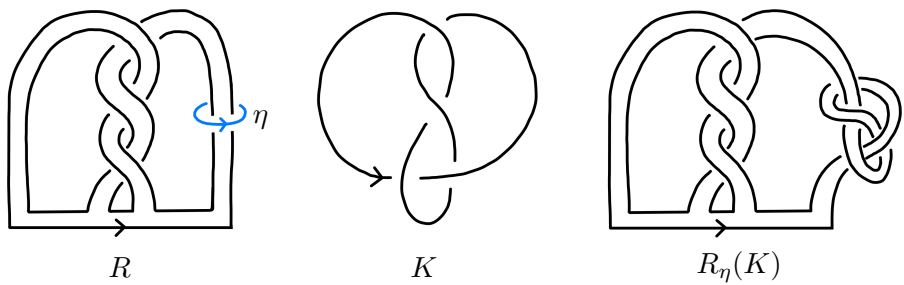
\begin{figure}[htb]
	\centering
\begin{tikzpicture}
		\node[anchor=south west,inner sep=0] at (0,0){	\includegraphics[width=12cm]{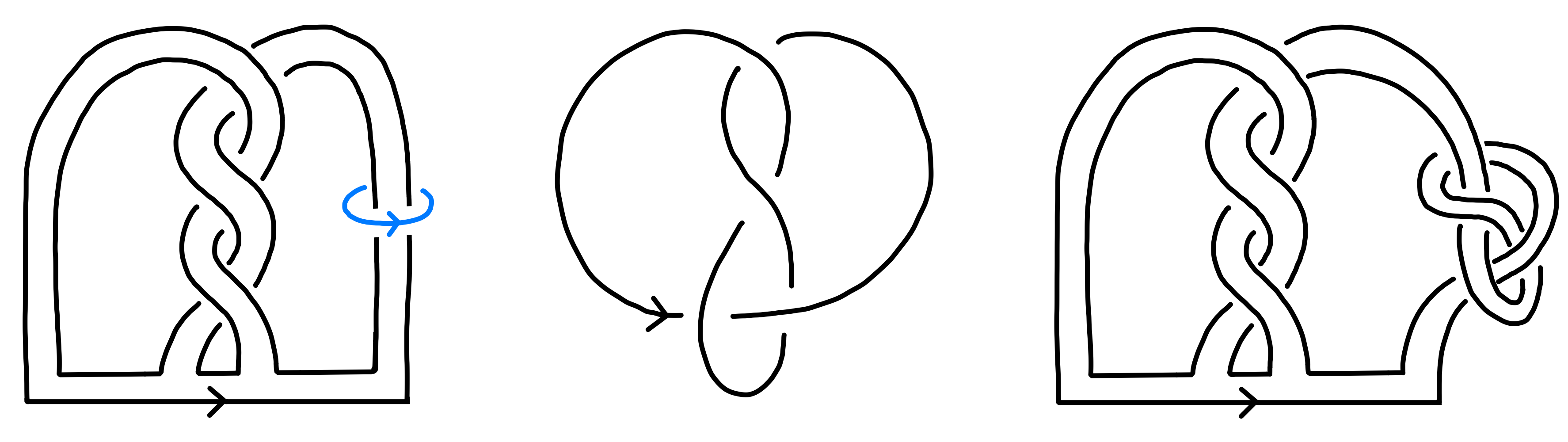}};
		%\draw[step=1cm,color=gray] (0,0) grid (12,3.3);	
		\node at (1.65,-0.2){$R$};		
		\node at (5.65,-0.2){$K$};
		\node at (3.5,1.75){$\eta$};
		\node at (9.8,-0.2){$R_\eta(K)$};
	\end{tikzpicture}	
  \caption{On the left, we have a ribbon knot $R$ with an unknot $\eta\subseteq S^3\sm R$. Given a knot $K$, we show on the right the result of infection on $R$ by $K$ along $\eta$.}\label{fig:infection}
\end{figure}

As previously indicated, the infection procedure can be used to construct knots in the various filtrations from the previous section. The key tool is the following proposition.

\begin{proposition}[\citelist{\cite{COT2}*{Proposition~3.1}\cite{CHL1}*{Theorem~7.1}\cite{CHL-fractal}*{Proposition~2.7}\cite{CHH}*{Proposition~3.3}}]\label{prop:solvable-infection}
Let $R\subseteq S^3$ be a knot and $\eta\subseteq S^3\sm R$ be a curve which is unknotted when considered in $S^3\supseteq S^3\sm R$. Let $K\subseteq S^3$ be an arbitrary knot, and let $n,i\geq 0$ be arbitrary.
\begin{enumerate}
\item if $R,K\in \F_n$ and $\eta\in \pi_1(S^3\sm R)^{(i)}$, then $R_\eta(K)\in \F_{n+i}$.
\item if $R,K\in \P_n$ and $\eta\in \pi_1(S^3\sm R)^{(i)}$, then $R_\eta(K)\in \P_{n+i}$.
\item if $R,K\in \N_n$ and $\eta\in \pi_1(S^3\sm R)^{(i)}$, then $R_\eta(K)\in \N_{n+i}$.
\item if $R,K\in \B_n$ and $\eta\in \pi_1(S^3\sm R)^{(i)}$, then $R_\eta(K)\in \B_{n+i}$.
\item if $R,K\in \T_n$ and $\eta\in \pi_1(S^3\sm R)^{(i)}$, then $R_\eta(K)\in \T_{n+i}$.
\end{enumerate}
\end{proposition}

We leave the proof for the ambitious reader (\orange{Exercise~$\square$}~\ref{ex:solvable-infection}). Armed with this proposition, we are now able to construct knots deep in any of the solvable, positive, negative, or bipolar filtrations, by the following strategy. Let us consider the solvable filtration for concreteness and attempt to build a knot in $\F_n$ for some given $n\in \NN$. Start with any knot $K$ with $\Arf(K)=0$, e.g.~the connected sum of two copies of the right-handed trefoil knot. Recall that we know that $K\in\F_0$ by \red{Exercise~$\bigcirc$}~\ref{ex:arf-0solvable}. Then choose a ribbon knot $R$. By~\cref{thm:T-in-solvable-int}, we know that $R\in \bigcap_{i\in \tfrac{1}{2}\NN} \F_i$. Find an unknotted $\eta\subseteq S^3\sm R$ such that $\lk(R,\eta)=0$. By definition of the linking number we know that $\eta\in \pi_1(S^3\sm R)^{(1)}$. So then $R_\eta(K)\in \F_1$ by \cref{prop:solvable-infection}. Then we iterate. In other words, to find a knot in $\F_n$, we need the $n$-fold iterated satellite $R_\eta(R_\eta(\cdots(K)\cdots))$. Note that these iterated satellites have bounded Seifert genus (\orange{Exercise~$\square$}~\ref{ex:satellite-genus}) and we only needed the $k=1$ case of the proposition. 

\textit{A priori} it might seem that satellites are quite special. However, by~\cite{CFT09}*{Proposition~1.7}, any algebraically slice knot is concordant to a knot of the form $R_\eta(J)$, where $J$ is a string link with pairwise linking numbers zero, the knot $R$ is ribbon, and $\eta$ is now an unlink, where each component has linking number zero with $R$. Therefore, the construction explained above is rather general. 

We finish this section by surveying a few results and open questions about the effect of satellite operators on the knot concordance groups. We already saw questions about the Whitehead doubling satellite operator in~\cref{sec:defns}. E.g.~\cref{conj:whitehead-injective} can be rephrased as asking whether the Whitehead doubling satellite operator is injective. 

\begin{question}[Open]
Is there a pattern $P$ with winding number $\neq \pm 1$ such that $P(K)$ is slice if and only if $K$ is slice, in either category? More generally, such that $P(K)$ and $P(J)$ are concordant if and only if $K$ and $J$ are concordant?
\end{question}

Examples of winding number $\pm 1$ satellite operators which are injective on $\C^\top$ were given in~\cite{CDR14}. From another perspective, one could ask when satellite operators are surjective on the concordance groups. For winding numbers $\neq \pm 1$ it is not difficult to find nonsurjective satellite operators using algebraic invariants~\cite{DR-groupactions}*{Proposition~3.1}. The first example of a nonsurjective satellite operator on $\C^\diff$ with winding number $\pm1$ was given by A.~Levine in~\cite{levine-nonsurjective}. Other work has considered the invertibility of satellite operators, such as~\cites{DR-groupactions,miller-piccirillo}. 

The existence of injective nonsurjective satellite operators is evidence towards the following conjecture. Further evidence was provided in~\cite{CHL-fractal} by showing that a large class of winding number zero satellite operators, called \emph{robust doubling operators}, are injective on large subgroups of $\C^\diff$. 

\begin{conjecture}[Open, Cochran--Harvey--Leidy~\cite{CHL-fractal}]
The knot concordance groups have the structure of a fractal.
\end{conjecture}

We say a set has a fractal structure if it admits \textit{self-similarities at arbitrarily small scales}, following~\cite{fractal}*{Definition~3.1}. The results mentioned above show that infinite classes of winding number $\pm 1$ satellite operators, as well as the robust doubling operators, are candidate self-similarities for $\C^\top$ and $\C^\diff$. In order to understand the notion of ``scale'' one needs a suitable metric on the knot concordance groups. This has been studied in \cites{CH14,CHP15,CHPR23}. One might also consider the behaviour of satellite operators under iteration, such as in \cites{ray-iterates,chen-iteratedmazur}.

Since knots under concordance form a group and satellite operators are maps on the concordance groups, it is tempting to assume that they are homomorphisms. However, this is far from the case. In~\cite{gompf-smoothconc}, Gompf showed that the Whitehead doubling operator is not a homomorphism on $\C^\diff$, despite being the zero map on $\C^\top$. Modern tools from Heegaard--Floer homology can be used to obstruct satellite operators on $\C^\diff$ from being homomorphisms~\cites{levine-nonsurjective,hedden-cabling}. In the topological category, obstructions can be obtained from Casson--Gordon invariants~\cite{miller-homomorphisms}. Most generally, we have the following conjecture. (For a recent, partial resolution, see~\cite{ANM-heddenconj}.)

\begin{conjecture}[Open, Hedden~\cites{MPIM16,BIRS16}]
Let $P\subseteq S^1\times D^2$ be a pattern. If the induced satellite operator $P\colon \C^\diff\to\C^\diff$ is a homomorphism, then it agrees with either the zero map, the identity map, or the reversal map (taking each knot to the concordance class of its reverse).
\end{conjecture}

Finally, there is substantial interest in understanding whether satellite operators preserve linear independence. We highlight two conjectures of Hedden and Pinz\'{o}n-Caicedo. 

\begin{conjecture}[Open, Hedden--Pinz\'{o}n-Caicedo~\cite{HPC-infiniterank}*{Conjecture~2}]
The image of every nonconstant satellite operator on $\C^\diff$ generates an infinite rank subgroup of $\C^\diff$. 
\end{conjecture}

The winding number $\neq 0$ case is not difficult to prove using algebraic invariants. \cite{HPC-infiniterank} provides a criterion guaranteeing that winding number zero satellite operators have infinite rank images, generalising~\cites{hedden-kirk,PC-rank}. 

\begin{conjecture}[Open, Hedden--Pinz\'{o}n-Caicedo~\cite{HPC-infiniterank}*{Conjecture~3}]
For any nonconstant winding number zero operator $P$, there exists a knot $K$ for which the set $\{P(nK)\mid n\in \Z\}$ has infinite rank in $\C^\diff$.
\end{conjecture}

The above was partially confirmed in~\cite{rankexpanding} using tools from Heegaard--Floer homology.

\section*{Epilogue}

There are many aspects of the field of slice knots and knot concordance that we were unable to cover in these few lectures. So we end these notes with a brief, but surely incomplete, overview of those topics. 

First, readers interested in more information on classical knot concordance should look at the excellent survey by Livingston~\cite{livingston-survey}. In particular, one will find there a much more detailed description of the algebraic concordance group and the Casson--Gordon invariants. Many recent developments in knot concordance, especially in the smooth setting, use Heegaard--Floer homology. Readers interested in those techniques should look at the excellent survey of Hom~\cite{hom-survey}. Obstructions to sliceness can also be obtained from Khovanov homology~\cite{rasmussen-slicegenus}.

Quite a bit of recent research has considered knots in 3-manifolds other than $S^3$ and/or sliceness in $4$-manifolds other than $B^4$. A notable highlight is the proposed strategy of Freedman--Gompf--Morrison--Walker~\cite{man-and-machine} to attack the 4-dimensional Poincar\'{e} conjecture: find a knot $K$ and a homotopy 4-ball $\mathcal{B}$ with $\partial \mathcal{B}=S^3$ such that $K$ is smoothly slice in $\mathcal{B}$ but not in $B^4$. This would imply that $\mathcal{B}$ is not diffeomorphic to $B^4$, disproving the smooth 4-dimensional Poincar\'e conjecture. It is worth noting that the following weaker question is also open: does there exist a knot $K\subseteq S^3$ which is slice in an integer homology 4-ball but not in $B^4$? However there do exist \emph{rationally slice} knots, such as the figure eight knot, i.e.~knots that are slice in a rational homology ball, but not necessarily slice in $B^4$. For more on this topic see e.g.~\cites{kawauchi-rational,cha-memoirs,HKPS-rational,HKP-rational}.

 More recent work of Manolescu and Piccirillo~\cite{manolescu-piccirillo-exotic} explains how one can use sliceness of knots and links to address questions about exotic smooth structures on closed 4-manifolds other than $S^4$. Other work concerning slicing knots in general 4-manifolds includes~\cites{norman-dehn,yasuhara-slice1,yasuhara-slice2,cochran-tweedy,conway-nagel,raoux-taurational,mmp-indefinite,pichelmeyer,klug-ruppik-deepshallow,aaas,mmrs,mmsw,hedden-raoux-adjunction}. When studying concordance in more general 3-manifolds, one could also quotient out by the action of connected sum by knots in $S^3$. This leads to the notion of \emph{almost concordance}, previously called \emph{piecewise linear $I$-equivalence}~\citelist{\cite{rolfsen-PLIequivalence}\cite{hillmanbook2012}*{{Section~1.5}}\cite{celoria-ac}\cite{fnop-ac}\cite{eylem}\cite{nopp}}.

In light of \cref{prop:0surgerychar} it is natural to ask to what extent the homeomorphism or homology cobordism class of the 0-surgery determines the concordance class of a knot. This is roughly the \emph{Akbulut--Kirby} conjecture, For more on this, see e.g.~\citelist{\cite{kirbylist}*{Problem~1.19}\cite{KL-mutation}\cite{CFHH}\cite{yasui-akbulut-kirby}\cite{miller-piccirillo}}. 

We introduced ribbon discs in \cref{sec:defns}. There is a relative version called \emph{ribbon concordance} in the smooth setting and \emph{homotopy ribbon concordance} in the topological setting (see also \orange{Exercise~$\square$}~\ref{ex:homotopyribbon}). These are not necessarily symmetric relations. A flurry of recent work in the smooth setting~\cites{zemke-ribbonconcordance,levine-zemke-ribbonconcordance,millerzemke-ribbonconcordance,DLVW-ribbonconcordance,kang-ribbonconcordance} culminated in Agol showing that ribbon concordance is a partial order on knots~\cite{agol-ribbonconc}, confirming a conjecture of Gordon~\cite{gordon-ribbonconcordance}. It is still open whether homotopy ribbon concordance is a partial order as well~\cites{friedl-powell-homotopyribbon,FKLNP-homotopyribbon}.

We are also interested in investigating sliceness and concordance within standard families of knots. For example, we saw in \cref{sec:defns} that the slice-ribbon conjecture holds for certain families. The families of algebraic knots~\cites{rudolph1976independent,litherland-cobordism,HKL12,ckp-algebraic}, 2-bridge knots~\cites{CG86,lisca-2bridge,miller-2bridge,feller-mccoy}, pretzel knots~\cites{greene-jabuka-pretzel,lecuona-pretzel,lecuona-montesinos,bryant-sliceribbon,long-sliceribbon,miller-pretzel,miller-pretzel-top,KST-alt-pretzel,KLS-pretzel}, and (strongly) quasipositive knots~\cites{rudolph:quasipositivity,hayden-cross,borodzik-feller} have received particular attention.

One may also study specific types of slice discs, e.g.~those that have certain symmetries. There has been a renewed interest recently in \emph{equivariant} sliceness and concordance~\cites{naik-equivariant,cha-ko-equivariant,davis-naik,boyle-issa,DMS-equivariant,boyle-chen,mallick-equivariant,miller-powell-equivariant,diprisa1,diprisa2,diprisa3}.

Finally so far we have only discussed the existence of slice discs. One could equally well study the uniqueness question. In other words, can one quantify the number of slice discs for a given slice knot? Recent work in this area includes~\cites{conway-powell-discs,sundberg-swann,juhasz-zemke-slice,miller-powell-discs,akbulut-ribbons,hayden-sundberg-surfaces,lipshitz-sarkar-slice}.

%%%%%%%%%%%%%%%%%%%%%%%%%%%

\clearpage
\section*{Exercises}
\subsubsection*{\green{Introductory problems}}

\begin{exercise-easy}\label{ex:knotexteriorhomologycircle}
    Let $K\subseteq S^3$ be a knot, with tubular neighbourhood $\nu K$.
    \begin{enumerate}
    \item Show that $S^3\sm \nu K$ is a homology circle, i.e.~$H_*(S^3\sm \nu K;\Z)\cong H_*(S^1;\Z)$.
    \item Show that $\pi_1(S^3\sm \nu K)$ is normally generated by an arbitrary meridian of $K$, i.e.~it is generated by the set of conjugates of the meridian.  
    \end{enumerate}
    Do the above properties generalise to higher-dimensional knots $S^n\hookrightarrow S^{n+2}$? How about knots with arbitrary codimension?
\end{exercise-easy}

\begin{exercise-easy}\label{ex:connsum-commutes}
Show that the connected sum operation is commutative and associative. In other words, given knots $J,K, L\subseteq S^3$, show that $J\#K$ is isotopic to $K\# J$ and $J \#(K\# L)$ is isotopic to $(J \#K)\# L$. 
\end{exercise-easy}

\begin{exercise-easy}\label{ex:reverse-mirror-slice}
Prove \cref{prop:reverse-mirror-slice}: Let $K\subseteq S^3$ be a knot. Show that $K$ is smoothly slice if and only if $rK$ is smoothly slice if and only if $\ol{K}$ is smoothly slice.

Note that the analogous statement also holds for topological sliceness.
\end{exercise-easy}

\begin{exercise-easy}\label{ex:conn-sum-slice}Prove~\cref{prop:conn-sum-slice}: 
If the knots $K,J\subseteq S^3$ are smoothly slice, then so is $K\# J$. 

As before, the analogous statement also holds for topological sliceness.
\end{exercise-easy}

\begin{figure}[htb]
	\centering
\begin{tikzpicture}
        \node[anchor=south west,inner sep=0] at (0,0){	\includegraphics[width=12cm]{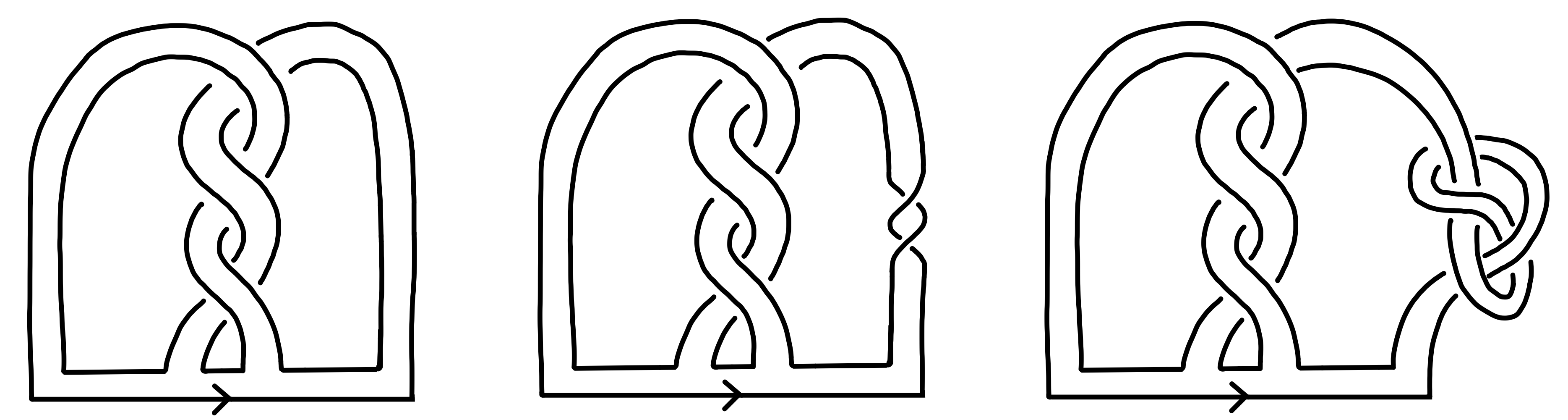}};
		%\draw[step=1cm,color=gray] (0,0) grid (12,3);
		\node at (1.65,-0.2) {(a)};
		\node at (5.55,-0.2) {(b)};
		\node at (9.5,-0.2) {(c)};
	\end{tikzpicture}
\caption{Some ribbon knots}\label{fig:ribbons}
\end{figure}
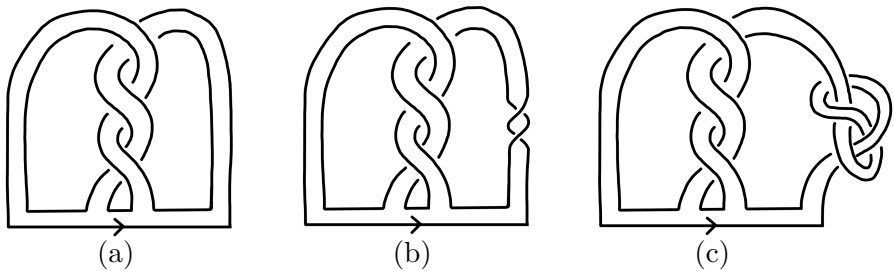 

\begin{exercise-easy}\label{ex:ribbon}
Show that the knots in~\cref{fig:ribbons} are ribbon. 
\end{exercise-easy}

\begin{exercise-easy}\label{ex:easy-ribbon}
Prove~\cref{prop:easy-ribbon}: For any knot $K\subseteq S^3$, the knot $K\# r\ol{K}$ is ribbon.
\end{exercise-easy}

\begin{exercise-easy}\label{ex:cone}
Prove \cref{prop:cone}: Let $K\subseteq S^3$ be a knot. The coned disc $\cone(K)\subseteq \cone(S^3)=B^4$ is locally flat if and only if $K$ is the trivial knot.

\textit{Hint:} Use the fact (without proof) from classical knot theory that $\pi_1(S^3\sm K)\cong \Z$ if and only if $K$ is the trivial knot.
\end{exercise-easy}

\begin{exercise-easy}\label{ex:wh-slice}
Let $K\subseteq S^3$ be a knot. Prove that if $K$ is smoothly or topologically slice, then the same holds for the untwisted Whitehead doubles $\Wh^\pm(K)$. 
\end{exercise-easy}

\begin{exercise-easy}\label{ex:slice-conc-unknot}
Prove~\cref{prop:slice-conc-unknot}: A knot $K\subseteq S^3$ is smoothly (resp.\ topologically) slice if and only if it is smoothly (resp.\ topologically) concordant to the unknot.
\end{exercise-easy}

\begin{exercise-easy}\label{ex:conc-equiv-rel}
Prove~\cref{prop:conc-equiv-rel}: Smooth (resp.\ topological) concordance is an equivalence relation on knots. 
\end{exercise-easy}

\begin{exercise-easy}\label{ex:easy-conc-facts}
Prove~\cref{prop:easy-conc-facts}: In either $\C^\diff$ or $\C^\top$, the inverse of $[K]$ is the class of $-K:=r\ol{K}$. The identity element in $\C^\diff$ (resp.\ $\C^\top$) is the class of the unknot, equivalently the class of any smoothly (resp.\ topologically) slice knot. 
\end{exercise-easy}

\begin{exercise-easy}\label{ex:neg-amph} 
Prove \cref{prop:neg-amph}: 
A knot $K\subseteq S^3$ is said to be \emph{negative amphichiral} if it is isotopic to $r\ol{K}$. Every negative amphichiral knot is order two in $\C^\diff$ and $\C^\top$. As a result, the figure eight knot is order two in $\C^\diff$ and $\C^\top$. 
\end{exercise-easy}

\begin{exercise-easy}\label{ex:solvable-subgroup}\label{ex:bipolar-subgroup}\leavevmode
\begin{enumerate}
\item For each $n\in\tfrac{1}{2}\NN$, show that the set $\F_n$ is a subgroup of $\C^\diff$. 
\item For each $n\in\NN$, show that the sets $\P_n$ and $\N_n$ are submonoids of $\C^\diff$, and $\B_n$ is a subgroup of $\C^\diff$. 
\end{enumerate}
\end{exercise-easy}

\begin{exercise-easy}
Let $K\subseteq S^3$ be a smoothly slice knot. Show that $K$ lies in $\T_n$ for all $n$. 
\end{exercise-easy}

\begin{exercise-easy}\label{ex:satellite-concordance}
Fix a pattern $P\subseteq S^1\times D^2$. Show that the satellite operation is well defined on concordance on either category, i.e.~there are well defined functions $P\colon \C^\diff\to \C^\diff$ and $P\colon \C^\top\to\C^\top$. 
\end{exercise-easy}

\begin{exercise-easy}\label{ex:homomorphism-slice}
Let $P\subseteq S^1\times D^2$ be a pattern so that the induced satellite operator $P\colon \C^\diff\to \C^\diff$ (resp.~$\C^\top\to \C^\top$) is a homomorphism. Show that $P(U)$ is smoothly (resp.~topologically) slice, where $U$ denotes the unknot. 
\end{exercise-easy}

\subsubsection*{\orange{Moderate problems}}
\begin{exercise-medium}\label{ex:Zknot}
Prove the hint from~\green{Exercise~$\triangle$}~\ref{ex:cone}, that $\pi_1(S^3\sm K)\cong \Z$ if and only if $K$ is the trivial knot. 

\textit{Hint:} Use Dehn's lemma.
\end{exercise-medium}

\begin{exercise-medium}
Learn enough classical knot theory to show that the right-handed trefoil, left-handed trefoil, and the figure eight from \cref{fig:exknots} are nontrivial and distinct knots. 

\textit{Hint:} This will likely involve learning about some classical knot invariants, such as $3$-colourability, the knot group, the Seifert genus, the signature, the Alexander polynomial, \dots. 
\end{exercise-medium}

\begin{exercise-medium}\label{ex:something}
Fix $n\geq 4$. Prove that every smooth $S^1\hookrightarrow S^n$ bounds a smoothly embedded disc in $S^n$. 

\textit{Hint:} The case $n=4$ is the most challenging. Think about the types of singularities that arise in the generic case and try to get rid of them.
\end{exercise-medium}

\begin{exercise-medium}\label{ex:ribbon-sing}
Prove \cref{prop:ribbon-sing}: A knot $K\subseteq S^3$ is ribbon if and only if it bounds a disc in $S^3$ with only ribbon singularities, i.e.\ singularities of the form shown in \cref{fig:ribbon-sing}.
\end{exercise-medium}

\begin{exercise-medium}\label{ex:homotopyribbon}
    A knot $K \subseteq S^3$ is \textit{homotopy ribbon} if there exists a topologically locally flat disc $D \subseteq B^4$ bounded by $K$ such that the inclusion induced map $\pi_1(S^3\setminus \nu K) \to \pi_1(B^4 \setminus \nu D)$ is surjective. Prove that ribbon implies homotopy ribbon. 
    
    \textit{Hint:} Consider handle decompositions for ribbon disc complements.
\end{exercise-medium}

\begin{exercise-medium}\label{ex:bad-disc}
Prove \cref{prop:bad-disc}: There exist smooth slice discs that are not ambiently isotopic (relative to the boundary) to any ribbon disc. 

\textit{Hint:} Begin with the standard smooth slice disc for the unknot. Use the fact that there exist $2$-knots in $S^4$ with nonabelian fundamental group of the complement (the intrepid reader could try to prove the latter claim). Use \orange{Exercise~$\square$}~\ref{ex:homotopyribbon}. For a further challenge, construct examples of such discs bounded by nontrivial knots.
\end{exercise-medium}

\begin{exercise-medium}\label{ex:trace-embedding-lemma}
Prove the trace embedding lemma (\cref{lem:trace-embedding-lemma}): Let $K\subseteq S^3$ be a knot. The $0$-trace $X_0(K)$ admits a smooth (resp. locally collared) embedding into $\R^4$ if and only if $K$ is smoothly (resp. topologically) slice. 

Wonder whether a similar argument would apply to the $n$-traces of the knots, denoted by $X_n(K)$, obtained by attaching an $n$-framed $2$-handle to $B^3$ along a knot $K\subseteq S^3=\partial B^4$, for an arbitrary $n\in \Z$.

\textit{Hint:} It is easier to prove the version of the trace embedding lemma in $S^4$ rather than $\R^4$. Avoid using the smooth $4$-dimensional Schoenflies conjecture (which is still open!) in the $S^4$ version by noting that the closure of the complement of a smoothly embedded $4$-ball in $S^4$ is itself diffeomorphic to a $4$-ball, by Palais's disc theorem.  
\end{exercise-medium} 

\begin{exercise-medium}\label{ex:conc-to-slice}
Prove~\cref{prop:conc-to-slice}: Two knots $K,J\subseteq S^3$ are smoothly (resp.\ topologically) concordant if and only if $K\# r\ol{J}$ is smoothly (resp.\ topologically) slice.

\textit{Hint:} It is possible to isotope a concordance so that it contains a straight arc, i.e.\ one of the form $\mathrm{pt} \times [0,1]\subseteq S^3\times [0,1]$ (why?). Remove an open tubular neighbourhood of this arc. For the other direction, build a concordance between $K\# r\ol{J}\# J$ and $J$, and use \green{Exercises}~\ref{ex:easy-ribbon} and~\ref{ex:slice-conc-unknot}.
\end{exercise-medium}

\begin{exercise-medium}\label{ex:iso-not-group}
Show that isotopy classes of knots do \textbf{not} form a group under connected sum. They do however form a monoid.
\end{exercise-medium}

\begin{exercise-medium}\label{ex:0surgerychar-traceembedding}
Reprove \cref{prop:0surgerychar} using the trace embedding lemma (\cref{lem:trace-embedding-lemma}).
\end{exercise-medium}

\begin{exercise-medium}\label{ex:solvable-infection}
Prove \cref{prop:solvable-infection}: Let $R\subseteq S^3$ be a knot and $\eta\subseteq S^3\sm R$ be a curve which is unknotted when considered in $S^3\supseteq S^3\sm R$. Let $K\subseteq S^3$ be an arbitrary knot, and let $n,i\geq 0$ be arbitrary 
\begin{enumerate}
\item if $R,K\in \F_n$ and $\eta\in \pi_1(S^3\sm R)^{(i)}$, then $R_\eta(K)\in \F_{n+i}$.
\item if $R,K\in \P_n$ and $\eta\in \pi_1(S^3\sm R)^{(i)}$, then $R_\eta(K)\in \P_{n+i}$.
\item if $R,K\in \N_n$ and $\eta\in \pi_1(S^3\sm R)^{(i)}$, then $R_\eta(K)\in \N_{n+i}$.
\item if $R,K\in \B_n$ and $\eta\in \pi_1(S^3\sm R)^{(i)}$, then $R_\eta(K)\in \B_{n+i}$.
\item if $R,K\in \T_n$ and $\eta\in \pi_1(S^3\sm R)^{(i)}$, then $R_\eta(K)\in \T_{n+i}$.
\end{enumerate}
\end{exercise-medium}

\begin{exercise-medium}\label{ex:satellite-genus}
Let $R\subseteq S^3$ be an arbitrary knot and $\eta\subseteq S^3\sm R$ be a curve which is unknotted when considered in $S^3\supseteq S^3\sm R$. Assume that $\lk(R,\eta)=0$. Show that the Seifert genus of the knots produced by infection on $R$ along $\eta$ is bounded. In other words, there exists $g$, such that for any knot $K$, the Seifert genus of $R_\eta(K)$ is at most $g$. Come up with a candidate $g$. 
\end{exercise-medium}
\subsubsection*{\red{Challenge problems}}

\begin{exercise-hard}
See~\orange{Exercise~$\square$}~\ref{ex:Zknot}. Are there any restrictions on the fundamental group of a slice/ribbon knot? In other words, is every knot group the knot group of a slice/ribbon knot? 

\textit{Hint:} Consider the Alexander module.
\end{exercise-hard}

\begin{exercise-hard}\label{ex:somethingelse}
Prove that every smooth $2$-knot is smoothly slice, i.e.\ every smooth $S^2\hookrightarrow S^4$ bounds a smoothly embedded $B^3$ in $B^5$. Prove that every topological $2$-knot is topologically slice. This was originally proven by Kervaire~\citelist{\cite{kervaire-2knots}*{Th\'{e}or\`{e}me~III.6}\cite{kervaire-odd}*{Theorem~1}}. 

\textit{Hint:} Use the fact that every $3$-manifold is obtained by even-framed surgery on some link in $S^3$. 
\end{exercise-hard}

\begin{exercise-hard}\label{ex:arf-0solvable}\label{ex:algslice-solvable}
Let $K\subseteq S^3$ be a knot. 
\begin{enumerate}
\item Show that $K\in \F_0$ if and only if $\Arf(K)=0$.
\item Show that $K\in \F_{0.5}$ if and only if $K$ is algebraically slice.
\item Show that if $K\in \F_{1.5}$ then all its Casson--Gordon sliceness obstructions vanish.
\end{enumerate}
\end{exercise-hard}

\clearpage

\def\MR#1{}
\bibliography{bib}
\end{document}